\documentclass{amsart}
\usepackage{amssymb,graphicx,amssymb,amsmath,lineno,amsopn}
\usepackage[ansinew]{inputenc}
\usepackage[T1]{fontenc}

\usepackage{textcomp}
\usepackage{charter}
\usepackage[]{bera}
\usepackage[charter]{mathdesign}

\usepackage{ifpdf}
\ifpdf
\usepackage[%
  pdftitle={Instructions for use of the document classelsart},%
  pdfauthor={Simon Pepping},%
  pdfsubject={The preprint document class elsart},%
  pdfkeywords={instructions for use, elsart, document class},%
  pdfstartview=FitH,%
  bookmarks=true,%
  bookmarksopen=true,%
  breaklinks=true,%
  colorlinks=true,%
  linkcolor=blue,anchorcolor=blue,%
  citecolor=blue,filecolor=blue,%
  menucolor=blue,pagecolor=blue,%
  urlcolor=blue]{hyperref}
\else
\usepackage[%
  breaklinks=true,%
  colorlinks=true,%
  linkcolor=blue,anchorcolor=blue,%
  citecolor=blue,filecolor=blue,%
  menucolor=blue,pagecolor=blue,%
  urlcolor=blue]{hyperref}
\fi

\makeatletter
\def\elsartstyle{%
    \def\normalsize{\@setfontsize\normalsize\@xiipt{14.5}}
    \def\small{\@setfontsize\small\@xipt{13.6}}
    \let\footnotesize=\small
    \def\large{\@setfontsize\large\@xivpt{18}}
    \def\Large{\@setfontsize\Large\@xviipt{22}}
    \skip\@mpfootins = 18\p@ \@plus 2\p@
    \normalsize
} \@ifundefined{square}{}{} \makeatother

\newtheorem{teo}{Theorem}
\newtheorem{proposition}{Proposition}
\newtheorem{cor}{Corollary}
\newtheorem{lema}{Lemma}
\theoremstyle{definition}

\theoremstyle{remark}

\newcommand{\brn}{\begin{eqnarray*}}
\newcommand{\ern}{\end{eqnarray*}}

\newcommand{\prodint}[1]{\left\langle{#1}\right\rangle}
\newcommand{\ov}[1]{\overline{#1}}

\title[Matrix biorthogonal polynomials, associated polynomials and functions of the second kind ]{Matrix biorthogonal polynomials, associated polynomials and functions of the second kind}

\begin{document}

\author[A. Branquinho]{Amilcar  Branquinho}
\address{CMUC 	and Department of Ma\-the\-ma\-tics, University of Coimbra, Apartado 3008, EC Santa Cruz, 3001-501 COIMBRA, Portugal.}
\email{ajplb@mat.uc.pt}
\thanks{AB acknowledges Centro de Matem\'{a}tica da Universidade de Coimbra (CMUC) -- UID/MAT/00324/2013, funded by the Portuguese Government through FCT/MEC and co-funded by the European Regional Development Fund through the Partnership Agreement PT2020.}

\author[J.C. García-Ardila]{Juan Carlos García-Ardila}
\address{Departamento de Matemáticas, Universidad Carlos III de Madrid, Avenida Universidad 30, 28911 Leganés, Spain}
\email{jugarcia@math.uc3m.es }

\author[F. Marcellán]{Francisco Marcellán}
\address{Departamento de Matemáticas, Universidad Carlos III de Madrid and Instituto de Ciencias Matemáticas (ICMAT), Avenida Universidad 30, 28911 Leganés, Spain}
\email{pacomarc@ing.uc3m.es }
\thanks{MM \& FM thanks financial support from the Spanish ``Ministerio de Economía y Competitividad" research project MTM2012-36732-C03-01,  \emph{Ortogonalidad y aproximación; teoría y aplicaciones}}

\subjclass{33C45, 42C05, 47A56 (primary), 41A10  (secondary)}
\keywords{ratio asymptotic; quadrature formulae; Markov functions; matrix biorthogonal polynomials; generalized Chebyshev polynomials}

 \begin{abstract}
In this work the interplay between matrix biorthogonal polynomials with respect to a matrix of linear functionals, the $k$-th associated matrix polynomials and the second kind matrix functions, is studied in terms of quasideterminants. A sort of Poincar\'{e}'s theorem for the ratio of two consecutive matrix functions solutions of a linear difference equation is also presented.
Some new formulas connecting these families of matrix functions are given.
 \end{abstract}

 \maketitle

\section{Introduction} \label{sec:1}

Let $\big\{p_n \big\}_{n\in\mathbb N} \, $ be a sequence of orthonormal polynomials with respect to a probability measure,~$\mu$, supported on an infinite subset of the real line. It is well known that $\big\{ p_n \big\}_{n\in\mathbb N} \, $ satisfies a three term recurrence relation
\begin{gather*}
% &
x \, p_n(x)=a_{n+1} \, p_{n+1}(x)+b_n \, p_n(x) + a_n \, p_{n-1}(x),
 \ \ n\geq 0 \, , %\\
% & p_0(x)=1 , \ \ p_{-1}(x)=0 \, ,
\end{gather*}
with initial conditions $p_0(x)=1$, $p_{-1}(x)=0$, where $\big(a_n\big)_{n\in\mathbb N}\, $ and $\big(b_n\big)_{n\in\mathbb N} \, $ are sequences of real numbers with $a_n\neq0 $.
If we assume that
$%\displaystyle \lim_{n \to \infty}
a_n %=
\to a \not = 0$ and
$%\displaystyle \lim_{n \to \infty}
b_n %=
\to b$,
 %with $a\neq 0 ,$,
then Nevai~\cite{Nevai} proved~that the convergence
\begin{gather}\label{du}
\lim_{n\to\infty}\frac{p_n(z)}{p_{n-1}(z)}=\frac{z-b+\sqrt{(z-b)^2-4a^2}}{2a} \, , \ \
z \in \mathbb C \setminus\operatorname{supp} \, (u) \, ,
\end{gather}
holds uniformly on compact subsets of $\mathbb C \setminus \operatorname{supp} \, (\mu)$.

The functions of second kind, $\big\{ q_n \big\}_{n\in\mathbb N} \, $, defined as
\begin{gather*}
q_n (z) =\int\frac{p_n(x)}{z-x} \, d\mu (x) \, , \ n \in \mathbb N \, ,
\ \ z \in \mathbb C \setminus\operatorname{supp} \, (u) \, ,
\end{gather*}
%then $\{ q_n \}_{n\in\mathbb N}$
%also
satisfy the same three term recurrence relation as $ \big\{ p_n \big\}_{n\in\mathbb N} \, $, but
%now
with initial conditions $q_{-1}(z)=%\frac
{1} \big/ {a_0}$,
$%\displaystyle
q_0(z)=\int %\frac
{ d\mu (x)} \big/ ({z-x} ) \, $.

 By the Poincar\'e's theorem for difference equations~\cite{poincare} it is possible conclude~that
\begin{gather}\label{eq11}
\lim_{n\to\infty} \frac{q_n(z)}{q_{n-1}(z)}
= \frac{z-b-\sqrt{(z-b)^2-4a^2}}{2a} \, , \ \ z \in \mathbb C \setminus\operatorname{supp} \, (u) \, ,
\end{gather}
uniformly on compact subsets of~$\mathbb C \setminus \operatorname{supp} \, (\mu)$. From this, Van Assche~\cite{walter} (see also~\cite{Rakhmanov}) showed~that
\begin{gather}\label{eq22}
\lim_{n\to\infty}{p_n(z)} \, {q_{n}(z)}=\frac{1}{\sqrt{(z-b)^2-4a^2}} \, , \ \
z \in \mathbb C \setminus\operatorname{supp} \, (u) \, ,
\end{gather}
and the convergence holds uniformly on compact subsets of $\mathbb C \setminus\operatorname{supp} \, (u) \, $.

Now, taking into account~\cite{Duran3}, we consider a positive definite $N \times N$ matrix of measures and its corresponding sequence of matrix orthonormal polynomials~$ \big\{ P_n \big\}_{n\in\mathbb N} \, $ satisfying a recurrence~relation
\begin{gather*}
%&
x \, P_n(x)=A_{n+1}\, P_{n+1}(x)+B_n \, P_n(x) + A^\top_n P_{n-1} \, , \ n \geq 0\, ,
% \\ &P_0(x)=I_N , \, P_{-1}(x)= \pmb 0 \, ,
\end{gather*}
with initial conditions $P_0(x)=I_N $, $ P_{-1}(x)= \pmb 0 $, where, $A_{n}$, are nonsingular matrices and~$B_{n}$ are Hermitian matrices; then the outer ratio asymptotics of two consecutive polynomials belonging to the {\it matrix Nevai class}, i.e.
if
$%\displaystyle \lim_{n \to \infty}
A_n %=
\to A$
and
$%\displaystyle
%\lim_{n \to \infty}
B_n %=
\to B$
with $A$ a nonsingular matrix, then
%\begin{gather*}
%\lim_{n\to\infty}
$\big\{ P_{n-1} \, P^{-1}_n \, A_n^{-1} \big\}_{n \in \mathbb N} \, $
uniformly converges on compact subsets of
$\mathbb C \setminus\Gamma^{(0)}$ to
$%\displaystyle
\int %\frac
{dW_{A,B}(y)}\big/ ({x-y})$,
%\end{gather*}
where $\Gamma^{(0)}$ will be defined later, cf.~\eqref{zer}, and
$W_{A,B}$ is the matrix weight for the Chebyshev matrix polynomials~(cf. \cite{Duran1}). A more general case was studied in~\cite[Theorem~\ref{teo4}]{AB_Paco_AM1} for matrix biorthogonal polynomials (cf. also~\cite{AB_Paco_JCGA}).

In the present contribution we are interested to analyze the analogous results in the matrix biorthogonal case (cf.~\cite{CAF_GA_JCGA_MM_FM} for a fresh introduction on matrix biorthogonal polynomials) of those given in~\eqref{eq11} and~\eqref{eq22}.

Observe that in this matrix scenario the Poincar\'e's theorem is no longer valid.
The answers to these problems are given as Corollaries of the main result of this paper (cf. Theorem~\ref{teorema 6}).
Here we present the result for matrices of linear functionals (and biorthogonal polynomials).

The structure of this manuscript is as follows. In Section~\ref{sectio2}, we exhibit the basic theory of linear difference equations on the noncommutative ring of matrices. In Section~\ref{sectio3}, we introduce the concept of matrix of linear functionals and its associated families of matrix biorthogonal polynomials.
Recall that the families of biorthogonal polynomials are solutions of a second order linear difference equation but these are no the unique ones. With this background, in Section~\ref{sectio4}, we present results concerning the independence of solutions for linear difference equations with matrix coefficients, as well as an explicit representation for these solutions.
In~Section~\ref{sectio5}, we study the
outer ratio asymptotics for the second kind matrix functions (see Theorem~\ref{teorema 6}) generalizing for the matrix case the results given in~\eqref{eq11} and~\eqref{eq22}.

\section {Linear difference equations on the
%non commutative
ring of matrices}\label{sectio2}

First of all we will fix some notation. Let $\mathbb{R}$ and $\mathbb{C}$ be the set of complex and real numbers, respectively, and denote by $\mathbb{C} ^{N\times N}$ (respectively, $\mathbb{R} ^{N\times N}$) the linear space of $N\times N$ matrices with complex entries (respectively, the linear space of~$N\times N$ matrices with real entries).

For an arbitrary finite or infinite matrix $A \, $, the matrix $A^\top$ is its transpose.
% of the matrix $A \, $.
%When $A$ is a block square matrix, $A_{n}$ means the principal leading $N \times N $ block submatrix of $A$ of order $n$.
The matrix $\pmb 0 $ will be understand as the null matrix of size~$N\times N$.

In the sequel, we will use the definition of quasideterminants coming from the last corner of the block matrix to obtain connection formulas between some families of orthogonal polynomials. They constitute a generalization of the determinants when the entries of the matrix belong to a noncommutative ring, and they share several properties with them.

Let $A\in\mathbb{C}^{M\times M}$, $B\in\mathbb{C}^{M\times N}$, $C\in\mathbb{C}^{N\times M}$ and $D\in\mathbb{C}^{N\times N}$, with $A$ a nonsingular matrix. For the $2\times 2$ block matrix
%\begin{gather*}
$\left(
\begin{matrix}
A& B \\ C & D
\end{matrix}
\right) \, $,
%\end{gather*}
the {\it last quasideterminants} is defined~by
\begin{gather*}
\Theta_*\left(\begin{matrix}A& B \\ C &D \end{matrix} \right):= D-CA^{-1}B \, .
\end{gather*}
Notice that the last quasideterminant is just the Schur complement of the block~$A \, $.
\begin{proposition}
\label{detpro}
Given the block matrix,
%\begin{gather*}
$\left(\begin{matrix}
A & B \\ C & D
\end{matrix}\right) \, $,
%\end{gather*}
where $A,B,C$, and $D$ are matrices of size $N\times N$,~then:
%\begin{enumerate}
%\item

$\phantom{ola}${\rm 1}. %\begin{gather*}
$%\displaystyle
\det \, \left(\begin{matrix} A&B \\ C&D\end{matrix} \right)
=\det \, \left( \begin{matrix} -C&-D\\A&B\end{matrix} \right) \, $.
%\end{gather*}

%\item

$\phantom{ola}${\rm 2}. If $A$ is
%a
nonsingular
%matrix,
then
%\begin{gather*}
$%\displaystyle
\det \, \left( \begin{matrix} A&B\\C&D \end{matrix} \right)
=\det \, (A) \, \det \, (D-C \, A^{-1}B) \, $.
%\end{gather*}

%\item
$\phantom{ola}${\rm 3}. If $A$ and $D-C \, A^{-1}B$ are nonsingular matrices, then
\begin{gather*}
\left(\begin{matrix} A&B\\C&D\end{matrix}\right)^{-1}=
\left( \begin{matrix} A^{-1}+A^{-1}B \, (D-C \, A^{-1} B)^{-1}C \, A^{-1}&-A^{-1}B \, (D-C \, A^{-1}B)^{-1} \\
-(D-C \, A^{-1}B)^{-1}C \, A^{-1}&(D-C \, A^{-1}B)^{-1}
 \end{matrix} \right) \, .
\end{gather*}
%\end{enumerate}
\end{proposition}
Recall that if $R$ is a ring, %then a
we say that a {\it left module over $R$} is a set $M$ together with two operations
\begin{gather*}
 \oplus : M\times M\rightarrow M \ \text{and} \ \
 \odot : R\times M\rightarrow M \, ,
\end{gather*}
such that for $m,n\in M$ and $a,b\in R$ we have:
%\begin{enumerate}

{\noindent}$\phantom{ola}$1. %\item
$(M, \oplus )$ is an Abelian group.
 \\	
$\phantom{ola}$2. %\item
$(a \oplus b) \odot m= ( a \odot m ) \oplus ( b \odot m )$ \ and \ \ $a \odot (m \oplus n)= ( a \odot m ) \oplus ( a \odot n ) $.
 \\
$\phantom{ola}$3. %\item
$(a \odot b) \odot m=a \odot (b \odot m)$.
%\end{enumerate}

In a similar way, one defines a right module on $R$. If $M$ is a left and right module over $R$, then $M$ is said to be a {\it bimodule}~(cf. \cite{lang,rowen}).

The {\it module} $M$ is said to be a {\it free left} module (respectively, {\it right} module) over $R$ if $M$ admits a basis, that is, there exists a subset $S$ of $M$ such that $S$ is not empty,~$S$ generates $M$, i.e. $M=\operatorname{span} \, (S)$ and $S$ is linearly independent.

 Recall that for matrices $A_k\in \mathbb{C}^{N \times N }$, $0\leqslant k \leqslant n$, with $\det \, (A_n)\neq0$, the matrix $P(x)=A_n \, x^n+A_{n-1} \, x^{n-1}+\cdots +A_1 \, x+A_0$ is said to be a matrix polynomial of degree $n $. In particular, if $A_n=I_N $, i.e. $A_n$ is the identity $N \times N $ matrix, then the polynomial is said to be monic. The set of matrix polynomials with coefficients in $\mathbb{C}^{N \times N }$ will be denoted by $\mathbb{C}^{N \times N }[x]$.

Observe that $\mathbb C^{N\times N}[x]$ with the usual sum and product for matrices is a free
bimodule and, in particular, a left module, on the ring $\mathbb{C}^{N \times N }$ with basis~$\big\{
%I_N, x \, I_N,
 x^n I_N
%, \ldots
\big\}_{n\in\mathbb N} \, $.
Important submodules of $\mathbb C^{N\times N}[x]$ are the sets $\mathbb C_n^{N \times N }[x]$ of matrix polynomials of degree less than or equal to $n$ with the basis
$\big\{I_N ,
%x \, I_N ,
\ldots , x^nI_N \big\}
%_{n\in\mathbb N}
$ of cardinality $n+1$.

The complex number, $x_0
%\in \mathbb{C}
$,
is said to be a zero of
$P %(x)
$
if $\det \, P(x_0) = 0$. Clearly, %from the above definition,
as a consequence, we have that $P
 %(x)
 $ has at most $n \, N$ zeros.

%\section {Linear difference equations on the noncommutative ring of matrices}

Now, we consider a sequence of %numerical
matrices $\big(A_{n}\big)_{n\in\mathbb N} \, $ in $\mathbb C^{N \times N }[x]$ and take the following $k$-th order difference equation
\begin{gather}\label{ecuacion}
y_{n+k}+A_{n+k-1} \, y_{n+k-1}+\cdots+A_{n} \, y_{n}={\bf 0}, \ \ n \in \mathbb N \, ,
\end{gather}
with initial conditions
%\begin{gather*}
$%\displaystyle
y_{0}=c_0,\ldots, y_{k-1}=c_{k-1} \, $,
%\end{gather*}
in  $\mathbb C^{N \times N }[x]$.  We denote by $\big(y_n(c)\big)_{n\in\mathbb N} \, $, which belong to $\mathbb C^{N \times N }[x]$, the solution of~\eqref{ecuacion} with the initial condition $c=\big(c_0,c_1,\ldots, c_{k-1}\big) $.
\begin{proposition}\label{unicidad}
The equation~\eqref{ecuacion} with initial condition $c$ has a unique~solution.
\end{proposition}
\begin{proof}
It is clear from the fact that given an initial condition $c$, $y_n(c)$ is completely determined.
\end{proof}

Now, we introduce the operator $\operatorname L \, $ as follows
\begin{gather*}
\operatorname L \, y_n=\sum_{i=0}^k A_{n+k-i} \, y_{n+k-i} \ \ \text{with} \ \ A_{n+k}=I_N \, .
\end{gather*}
Observe that if
$\, \big( y_n \big)_{n\in\mathbb N} \, $
is a solution of~\eqref{ecuacion}, then $\operatorname L \, y_n = \pmb 0 $. It is easy verify that the operator $\operatorname L \, $ is a {\it right linear operator}, i.e.
\begin{gather*}
\operatorname L \, (y_n \, \alpha+z_n \, \beta)=( \operatorname L \, y_n) \, \alpha+(\operatorname L \, z_n) \, \beta \, , \ \ \alpha,\beta\in \mathbb C^{N \times N }[x] \, .
\end{gather*}
%i.e., $\operatorname L$ is a right linear operator.
%\begin{definition}
%The set of solutions of~\eqref{ecuacion} will be denoted by $\mathrm{S} \, $.
We will denote by~$ \, \mathbb S \, $ the set of solutions of~\eqref{ecuacion}.
%\end{definition}

Notice that
if $\big(y_n\big)_{n\in\mathbb N} \, $ and $\big(z_n\big)_{n\in\mathbb N} \, $ are solutions of~\eqref{ecuacion}, then since the operator~$\operatorname L \, $ is right linear,
%from the fact~that
%\begin{gather*}
%\operatorname L (y_n \, \alpha+z_n \, \beta)=(\operatorname L y_n) \, \alpha+(\operatorname L z_n) \, \beta=\pmb 0 \, , \ \ \alpha,\beta\in \mathbb C^{N \times N }[x] \, ,
%\end{gather*}
we have $y_n \, \alpha + z_n \, \beta$ is also a solution of~\eqref{ecuacion}.

Moreover, it is clear that $ \, \mathbb S \, $ is an Abelian group under addition and since \begin{gather*}
y_n \, ( \alpha+\beta )= y_n \, \alpha + y_n \, \beta \, \ \mbox{ and } \ \
(y_n \, \alpha ) \, \beta = y_n \, ( \alpha \, \beta ) \, ,
\end{gather*}
 then we can conclude that $\, \mathbb S \, $ is also a right module over the noncommutative ring $C^{N \times N }[x] $.

\begin{proposition}\label{representation}
Let $ \big(y_n(e_0)\big)_{n\in\mathbb N} \, , \, \big(y_n(e_1)\big)_{n\in\mathbb N} \, , \, \ldots \, , \, \big(y_n(e_{k-1})\big)_{n\in\mathbb N} \, $ be the solution of~\eqref{ecuacion} with the initial conditions
\begin{gather*}
e_0=(I_N ,0,\ldots,0),\ e_1=(0,I_N ,\ldots, 0),\ldots, e_{k-1}=(0,0,\ldots, I_N ) \, .
\end{gather*}
Given a solution of~\eqref{ecuacion}, $\big(y_n(c)\big)_{n\in\mathbb N} \, $, with the set of initial conditions given by $c=\big( c_0 \cdots c_{k-1}\big) \, $, then 	
$\big(y_n(c)\big)_{n\in\mathbb N} \,$ can be expressed as a linear combination of $\big(y_n(e_i)\big)_{n\in\mathbb N} \, $, $i=0,\ldots, k-1\, $.
\end{proposition}
\begin{proof}
Let
%\begin{gather*}
$%\displaystyle
z_n=\sum_{i=0}^{k-1} y_n(e_i) \, c_i \, $.
%\end{gather*}
Since $ \, \mathbb S \, $ is a right module on $\mathbb C^{N \times N }[x] $, then
$%\displaystyle
\sum_{i=0}^{k-1}y_n(e_i) \, c_i \in \mathbb S \, $.
Moreover, observe that $z_i=c_i $, for $i=0,\ldots, k-1\, $. The above implies that $\big( z_n \big)_{n\in\mathbb N} \, $ is a solution of~\eqref{ecuacion} with initial condition~$c$, but from Proposition~\ref{unicidad}, $z_n=y_n$ for every $n\in\mathbb N \, $.
\end{proof}

%\begin{definition}
Given a set of functions $ \big\{ f_{i,n} \big\}_{n\in\mathbb N} \, $, $i=0,\ldots, k -1$,
%\begin{gather*}
$ f_{i,n}:\mathbb{N}\to \mathbb C^{N \times N }[x] \, $.
%\end{gather*}
The {\it set of functions} $ \big\{ f_{i,n} \big\}_{n\in\mathbb N} \, $,
$i=0,\ldots, k-1\, $,
are said to be {\it linearly independent} if for all $n\in\mathbb N \, $,
\begin{gather}\label{independencia}
\sum_{i=0}^{k-1} f_{i,n} \, \alpha_i=\pmb 0 \, , \ \ \alpha_i\in C^{N \times N }[x], \ \ \text{implies} \ \ \alpha_i=\pmb 0 \, .
\end{gather}
%\end{definition}

\begin{cor}\label{coro1} The set of solutions
$\big\{ \big(y_n(e_i) \big)_{n\in\mathbb N}:i=0,\ldots, k-1 \big\} \, $ is a basis for~$\, \mathbb S \, $, or equivalently,~$ \, \mathbb S \, $ is a free right module.	
\end{cor}
\begin{proof}
This fact follows from Proposition~\ref{representation}.
\end{proof}

%\begin{definition}
Given a set of functions $ \big\{ f_{i,n} \big\}_{n\in\mathbb N} \, $, $i=0,\ldots, k-1 $, we define the {\it block Casorati matrix} as
\begin{gather*}
W(f_{0,n}, \ldots , f_{k-1,n})=\left( \begin{matrix}
f_{0,n}&
%f_{1,n} &
\cdots & f_{k-1,n}\\
%f_{0,n+1}&f_{1,n+1}& \cdots & f_{k-1,n+1}\\
\vdots &
%\vdots &
\ddots & \vdots\\
f_{0,n+k-1} &
%f_{1,n+k-1}&
\cdots & f_{k-1,n+k-1}	
 \end{matrix} \right) \, .
\end{gather*} 	
%$W(f_{0,n}, \ldots , f_{k-1,n})$ is said to be a block Casorati matrix.
%\end{definition}

\begin{teo}\label{teorema-indepen}
A sufficient condition for the set of functions $\big\{ f_{i,n} \big\}_{n\in\mathbb N} \, $, $ i =0,\ldots, k-1$ be linearly independent is that there exists
$\hat n \in \mathbb N \, $ such that
\begin{gather*}
%$
\det \, W(f_{0,\hat n}, \ldots , f_{k-1,\hat n})\ne 0 \, .
%$.
\end{gather*}
\end{teo}

\begin{proof}
If~\eqref{independencia} holds for some $\hat n \in \mathbb N \, $, then
\begin{gather*}
\left( \begin{matrix}
\pmb 0 \\
%\pmb 0 \\
\vdots\\
\pmb 0
 \end{matrix} \right)
=\left( \begin{matrix}
f_{0,\hat{n}} &
%f_{1,\hat{n}} &
\cdots & f_{k-1,\hat{n}} \\
%f_{0,\hat{n}+1} & %f_{1,\hat{n}+1} & \cdots & f_{k-1,\hat{n}+1} \\
\vdots &
%\vdots &
\ddots & \vdots \\
f_{0,\hat{n}+k-1} &
%f_{1,\hat{n}+k-1} &
\cdots & f_{k-1,\hat{n}+k-1}
 \end{matrix} \right)
\left( \begin{matrix}
\alpha_1 \\
%\alpha_2\\
\vdots \\
\alpha_k
 \end{matrix} \right) \, .
\end{gather*}
The above system has a unique solution if and only if $W(f_{1, \hat n}, \ldots , f_{k,\hat n})$ is a nonsingular matrix~(see~\cite{hor1}).
\end{proof}

\section{Matrix biorthogonal polynomials}\label{sectio3}

%\begin{definition}
A {\it sesquilinear form}
%$\prodint{ \cdot , \cdot } $
on the bimodule $\mathbb C^{N \times N }[x]$, with real variable, is a~map
\begin{align*}
\prodint{\cdot,\cdot}: \mathbb{C}^{N \times N }[x]\times\mathbb{C}^{N \times N }[x]\to \mathbb{C}^{N \times N} \, ,
\end{align*}
such that for any triple $P, Q, R\in \mathbb{C}^{N \times N }[x]$ of matrix polynomials we have for~all~$A,B\in\mathbb{C}^{N \times N}$:

%\begin{enumerate}
{\noindent}$\phantom{ola}${\rm 1}. %\item
$\prodint{A \, P(x)+B \, Q(x),R(x)}=A\prodint{P(x),R(x)}+B\prodint{Q(x),R(x)}$; %for all $A,B\in\mathbb{C}^{N \times N }$.
 \\
$\phantom{ola}${\rm 2}. %\item
$\prodint{P(x),A \, Q(x)+B \, R(x)}=\prodint{P(x),Q(x)}A^\top+\prodint{P(x),R(x)}B^\top$. %for all $A,B\in\mathbb{C}^{N \times N }$.
%\end{enumerate}

If $\prodint{P(t),Q(t)}=\prodint{Q(t),P(t)}^{\top}$, %then
$\prodint{\cdot,\cdot}$ is called a {\it symmetric sesquilinear~form}.
%\end{definition}

 %\begin{definition}\label{distributional}
Given a matrix of linear functionals, i.e.
\begin{align*}
u=\left( \begin{matrix}
u_{1,1} & \cdots &u_{1,p}\\
\vdots & \ddots &\vdots \\
u_{p,1} & \cdots & u_{p,p}
\end{matrix} \right) \, ,
\end{align*}
where $u_{i,j} \, $ belong to the the dual space of $\mathbb C[x] \, $, we define its associated sesquilinear form $\prodint{P,Q}_u$ as follows
\begin{align*}
\big(\prodint{P,Q}_u\big)_{i,j}:=\sum_{k,l=1}^p \prodint{u_{k,l},P_{i,k}(x) \, Q_{j,l}(x)} \, .
\end{align*}
In this case the support is defined as
$ \operatorname{supp} \, (u):=
\bigcup_{k,l=1}^p \operatorname{supp} \, (u_{k,l}) \, $.
%\end{definition}

An important property of the sesquilinear form defined %from
in terms of a matrix of linear functional is that $\prodint{x \, P(x),Q(x)}=\prodint{P(x) , x \, Q(x)} $.

Let $\big\{ V_n \big\}_{n\in\mathbb N} \, $ and $ \big\{ G_n \big\}_{n\in\mathbb N} \, $ be two sequences of matrix polynomials~satisfying
\begin{gather*}
\prodint{V_n(x),G_m(x)}_u = I_N \, \delta_{n,m} \, , \ n,m \in \mathbb N \, .
\end{gather*}
The sequences of matrix polynomials $ \big\{ V_n \big\}_{n\in\mathbb N} \, $,
%and
$\big\{ G_n \big\}_{n\in\mathbb N} \, $ are said to be {\it biorthogonal with respect~to~$u \, $}.

It is well known, cf. for instance~\cite{AB_Paco_AM2}, that for the sequences of matrix polynomials~$ \big\{ V_n \big\}_{n\in\mathbb N} \, $ and~$\big\{ G_n \big\}_{n\in\mathbb N} \, $ there exist sequences of %numerical
matrices~$ \big(A_n \big)_{n\in\mathbb N} \, $, $\big(B_n \big)_{n\in\mathbb N} \, $, and~$ \big(C_n \big)_{n\in\mathbb N} \, $, with $A_n$ a lower triangular matrix and~$C_n$ a upper triangular matrix, both nonsingular for $n=0,1,2, \ldots $, such~that
 \begin{gather}
\label{eq.1}
x \, V_n (x) = A_n \, V_{n+1}(x) + B_n \, V_n (x) + C_n \, V_{n-1} (x) \, , \\
\label{eq.2}
x \, G^{\top}_n (x) = G^{\top}_{n+1}(x) \, C_{n+1} + G^{\top}_n (x) \, B_n + G^{\top}_{n-1} (x) \, A_{n-1} \, ,
\end{gather}
with initial conditions
$V_0(x)= G^{\top}_0(x)=I_N $ and $V_{-1}(x)= G^{\top}_{-1}(x)=\pmb 0 $.

Observe that from definition, $C_0=I_N $. In the same way $A_{-1}=I_N $. Moreover, the converse is also true, i.e. if we have two sequences of matrix polynomials $ \big\{ V_n \big\}_{n\in\mathbb N} \, $ and $ \big\{ G_n \big\}_{n\in\mathbb N} \, $ satisfying~\eqref{eq.1} and~\eqref{eq.2}, respectively, then there exists a matrix of linear functionals $u$, with respect to they are~biorthogonal.
%with respect to $u$.

Now, from the sequences of %numerical
matrices $\big(A_n\big)_{n \in \mathbb N} \, $, $\big( B_n \big)_{n \in \mathbb N} \, $, and $ \big(C_n \big)_{n \in \mathbb N} \, $ with $A_0=I_N $, one defines for each $n\in \mathbb N \, $ the {\it $k$-th associated polynomials $ \big\{ V_n^{(k)} \big\}_{n\in\mathbb N} \, $}
and
$\big\{ G^{(k)}_{n} \big\}_{n\in\mathbb N} \, $ by the recurrence formula, for~$n \in \mathbb N \, $,
\begin{gather}\label{reasociados}
 x \, V^{(k)}_n(x)=A_{n+k} \, V^{(k)}_{n+1}(x)+B_{n+k} \, V^{(k)}_n(x) + C_{n+k} \, V^{(k)}_{n-1}(x) \, , \\
%\end{gather}
%\begin{align*}
x \, G^{(k)\top}_n (x) = G^{(k)\top}_{n+1}(x) \, C_{n+k+1} + G^{(k)\top}_n (x) \, B_{n+k} + G^{(k)\top}_{n-1} (x) \, A_{n+k-1} \, , \notag
\end{gather}
with initial conditions
$V^{(k)}_{-1}(x)={\bf 0} \, $, $V^{(k)}_0(x)=I_N $ and $G^{(k)}_{-1}(x)={\bf 0} \, $, $G^{(k)}_0(x)=I_N$.

In~\cite{AB_Paco_AM1} is proved that for every $n\in \mathbb N \, $, $V^{(k)}_n (x)$ and $G^{(k)}_{n}(x)$ have the same zeros. Moreover, taking the $N$-block Jacobi matrix associated with the recurrence relation~\eqref{reasociados}, i.e.
\begin{gather*}
J^{(k)}=\left(\begin{matrix}
B_k & A_k & \pmb 0 & \\
C_{k+1} & B_{k+1} & A_{k+1} & \ddots \\
\pmb 0 & C_{k+2}&B_{k+2} & \ddots \\
& \ddots & \ddots & \ddots
\end{matrix}\right) \, ,
\end{gather*}
$J^{(0)}=J$, the zeros of $V^{(k)}_n(x)$ are the eigenvalues of $J^{(k)}_n$, where $J_n$ is the truncated matrix of $J^{(k)}$ with size $n N\times nN$. So,
denoting by
%give the following definition,
$\Delta^{(k)}_{n}$
%will denote
the set of zeros of $V^{(k)}_n(x)$ (or equivalently, of $G^{(k)}_n(x)$), we~define
\begin{gather}\label{zer}
\Gamma^{(k)}= \bigcap_{n \in \mathbb N} M^{(k)}_N \ \ \text{where} \ \ M^{(k)}_N=\ov{\bigcup_{n \in \mathbb N} \Delta^{(k)}_n} \, .
\end{gather}
If $A_n$, $B_n $, and $C_n$ converge,
%(as is our case),
then by the Gershgorin disk theorem, there exists a real number $M>0$ such that $\ov{\bigcup_{n \in \mathbb N} \Delta_n}\subset \pmb{\operatorname{D}}_M$, where $\pmb{\operatorname{D}}_M=\big\{x: |x| \leq M \big\} $. Moreover, since $J^{(k)}$ is a submatrix of $J$, then again from the Gershgorin disk theorem, we have $\bigcup_{k \in \mathbb N} \Gamma^{(k)}\subset\pmb{\operatorname{D}}_M$.

%\begin{definition}
For $y\notin \operatorname{supp} \, (u)$, the corresponding families of {\it second kind functions}, $\big\{Q_n \big\}_{n\in\mathbb N} \, $ and~$\big\{ R_n \big\}_{n\in\mathbb N} \, $, are defined by
\begin{align*}
Q_n (y) = \prodint{V_n (x), \frac{I_N }{y-x}}_{u} \, \ \mbox{ and } \ \ R_n^{\top}(y)= \prodint{\frac{I_N }{y-x},G_n (x)}_u \, .
\end{align*}
%\end{definition}

%\begin{remark}
Observe that the families of second kind functions $\big\{ Q_n \big\}_{n\in\mathbb N} \, $ and
$\big\{ R_n \big\}_{n\in\mathbb N} \, $ also satisfy the following three term recurrence relations
\begin{gather*}
y \, Q_n (y) =  A_n \, Q_{n+1}(y) + B_n \, Q_n (y) + C_n \, Q_{n-1}(y) \, , \\
y \, R^{\top}_n (y)
 =  R^{\top}_{n+1}(y) \, C_{n+1} + R^{\top}_n (y) \, B_n + R^{\top}_{n-1} (y) \, A_{n-1} \, ,
\end{gather*}
with
%initial conditions
$%\displaystyle
Q_0(y)=\prodint{I_N ,\frac{I_N }{y-x} }_{u} $,
$%\displaystyle
Q_{-1}(y)=C_0^{-1} $,
$%\displaystyle
R^{\top}_0(y)=\prodint{\frac{I_N }{y-x},I_N}_{u} $,
and
\linebreak
$%\displaystyle
R^{\top}_{-1}(y)=A^{-1}_{-1} $.
%\end{remark}
\begin{proposition}[Christoffel-Darboux type formulas] \label{pro:cdt}
Let $ \big\{ V_n \big\}_{n\in\mathbb N} \, $ and $\big\{ G_n \big\}_{n\in\mathbb N} \, $ be the sequences of biorthogonal polynomials with respect to $u \, $. Let $ \big\{Q_n \big\}_{n\in\mathbb N} \, $ and~$ \big\{ R_n \big\}_{n\in\mathbb N} \, $ be, respectively, the corresponding families of second kind functions,~then
%of $\big\{ V_n \big\}_{n\in\mathbb N}$ and $ \big\{ G_n \big\}_{n\in\mathbb N} $,
\begin{gather}
%{multline}
\label{1C}
\phantom{0}
\hspace{-.75cm}
(x-y)\sum_{m=0}^n G^{\top}_m(y) \, Q_m(x)
    %\\
=G^{\top}_n(y) \, A_n \, Q_{n+1}(x) - G^{\top}_{n+1}(y) \, C_{n+1} \, Q_{n}(x) + I_N \, ,
\end{gather}
%{multline}
and its confluent formula
\begin{gather}\label{2C}
\sum_{m=0}^nG^{\top}_m(x) \, Q_m(x)=(G^{\top}_{n+1}(x))^{\prime} \, C_{n+1} \, Q_{n}(x)-(G^{\top}_n(x))^{\prime} \, A_n \, Q_{n+1}(x) \, .
\end{gather}
%In the same way
Moreover, we get the analogous Christoffel-Darboux and confluent formulas
\begin{gather} \label{3C}
(x-y) \sum_{m=0}^n R^{\top}_m(y) \, V_m(x) =R^{\top}_n(y) \, A_n \, V_{n+1}(x)-R^{\top}_{n+1}(y) \, C_{n+1} \, V_n(x)-I_N \, ,
  \\
%\end{multline}
%and its confluent one
%\begin{gather}
\label{4C}
\sum_{m=0}^n R^{\top}_m(x) \, V_m(x)
 =
R^{\top}_n(x) \, A_n \, V^{\prime}_{n+1}(x) - R^{\top}_{n+1}(y) \, C_{n+1} \, V^{\prime}_n(x) \, .
\end{gather}
\end{proposition}
\begin{proof}
We only prove~\eqref{1C} and~\eqref{2C}. The formulas~\eqref{3C} and~\eqref{4C} follow in a similar way. From recurrence formulas for $\big\{ Q_m \big\}_{m \in \mathbb N} \, $ (respectively, $\big\{ G_m \big\}_{m \in \mathbb N}$) multiplied on the left by $G_m^{\top}$ (respectively, multiplied on the right by $\big\{ Q_m \big\}_{m \in \mathbb N}$), we have
\begin{align}
%{multline}
 \notag
 & x \, G^{\top}_m (y) \, Q_m (x)
 \\
\label{2}
 & \phantom{ola}
= G^{\top}_m (y) \, A_m \, Q_{m+1}(x)
+ G^{\top}_m (y) \, B_m \, Q_m (x) + G^{\top}_m (y) \, C_m \, Q_{m-1}(x) \, , \\
%\end{multline}
%\begin{multline}
\notag
 & y \, G^{\top}_m (y) \, Q_m (x)
 \\
\label{1.5}
 & \phantom{ola}
= G^{\top}_{m+1}(y) \, C_{m+1} \, Q_m (x)
 + G^{\top}_m (y) \, B_m \, Q_m (x) + G^{\top}_{m-1} (y) \, A_{m-1} \, Q_m (x) \, .
\end{align}
%{multline}
From here, if we subtract~\eqref{1.5} from~\eqref{2}
\begin{multline*}
(x-y) \, G^{\top}_m (y) \, Q_m (x)=
\big(
G^{\top}_m (y) \, A_m \, Q_{m+1}(x) - G^{\top}_{m-1} (y) \, A_{m-1} \, Q_m (x)
\big)
 \\
- \big(
G^{\top}_{m+1}(y) \, C_{m+1} \, Q_m (x) - G^{\top}_m (y) \, C_m \, Q_{m-1}(x)
 \big) \, .
\end{multline*}
Summing the later from $0$ to $n$ and taking into account the initial conditions, the result follows. To show the confluent formula notice that
\begin{gather*}
G^{\top}_n(y) \, A_n \, Q_{n+1}(y) - G^{\top}_{n+1}(y) \, C_{n+1} \, Q_{n}(y) + I_N = \pmb 0 \, .
\end{gather*}
Now, observe that the Christoffel-Darboux formula can successively be rewritten~as
\begin{align*}
 & \sum_{m=0}^n G^{\top}_m(y) \, Q_m(x)
=
\frac{G^{\top}_{n+1}(y) \, C_{n+1} \, Q_{n}(y)-G^{\top}_n(y) \, A_n \, Q_{n+1}(y)+C_0^{-1}}{x-y} %\phantom{olaolaola}
     \\
& \phantom{olaolaolao} +G^{\top}_n(y) \, A_n \, \frac{Q_{n+1}(x)-Q_{n+1}(y)}{x-y}
-G^{\top}_{n+1}(y) \, C_{n+1}\frac{Q_{n}(x)-Q_n(y)}{x-y}
 \, , \\
%\end{multline*}
%i.e.
%\begin{multline*}
& %\sum_{m=0}^nG^{\top}_m(y) \, Q_m(x)
 \phantom{olaolao}=
G^{\top}_n(y) \, A_n \frac{Q_{n+1}(x)-Q_{n+1}(y)}{x-y}-G^{\top}_{n+1}(y) \, C_{n+1}\frac{Q_{n}(x)-Q_n(y)}{x-y} \, ;
\end{align*}
and taking $x\to y$ we get the desired result.
\end{proof}
Now, we state a sort of reciprocal of Proposition~\ref{pro:cdt}.
\begin{proposition}
Suppose that we have two sequences, $\big\{ V_n \big\}_{n\in\mathbb N} \, $ and $ \big\{ G_n \big\}_{n\in\mathbb N} \, $, of matrix polynomials, and two matrices of linear functionals $u^1$, %and~
$u^2$ such that
$%\displaystyle
\prodint{V_n,I_N }_{u^1}=\delta_{n,0}$ and
$%\displaystyle
\prodint{I_N ,G_n}_{u^2}=\delta_{n,0}$, $n \in \mathbb N \, $. We define 
the matrix functions
$%\displaystyle
Q_n (y) = \prodint{V_n (x), \frac{I_N }{y-x}}_{u^1}$,
$%\displaystyle
R_n^{\top} (y)= \prodint{\frac{I_N }{y-x},G_n (x)}_{u^2} $ and we assume that~\eqref{1C} and~\eqref{3C} are satisfied, then $u^1 \equiv u^2 $.
\end{proposition}
\begin{proof}
Observe that from the confluent formula
\begin{gather*}
\sum_{m=0}^nG^{\top}_m(x) \, Q_m(x)=\sum_{m=0}^{n-1}G^{\top}_m(x) \, Q_m(x) +G^{\top}_n(x) \, Q_n(x) \, ;
\end{gather*}
we get using~\eqref{1C}
\begin{multline*}
G^{\top}_n(x) \, Q_n(x)=\big(
(G^{\top}_{n+1}(x))^{\prime} \, C_{n+1}+(G^{\top}_{n-1}(x))^{\prime} \, A_{n-1}
\big) \, Q_{n}(x)
 \\
-(G^{\top}_n(x))^{\prime} \, \big(A_n \, Q_{n+1}(x)+C_{n} \, Q_{n-1}(x)\big) \, .
\end{multline*}
or, equivalently,
\begin{multline}\label{rew}
I_N =G^{-\top}_n(x) \, \big(G^{\top}_{n+1}(x) \, C_{n+1} + G^{\top}_{n-1}(x) \, A_{n-1}\big)^{\prime} \, Q_{n}(x)\\
-G^{-\top}_n(x) \, (G_n^{\top}(x))^{\prime} \, \big(A_n \, Q_{n+1}(x)+C_{n} \, Q_{n-1}(x)\big) \, Q^{-1}_n(x) \, .
\end{multline}
Now, for every $n\in\mathbb N \, $,
\begin{gather}\label{ide}
I_N =G^{\top}_{n+1}(x) \, C_{n+1} \, Q_n(x)-G^{\top}_{n}(x) \, A_{n+1} \, Q_{n+1}(x) \, ,
\end{gather}
 and as
%\begin{gather*}
%\displaystyle
$(G^{-\top}_n(x))^{\prime}=-G^{-\top}_n(x) \, (G^{\top}_n(x))^{\prime}\, G^{-\top}_n(x) \, $,
%\end{gather*}
we can rewrite~\eqref{rew} as~follows
\begin{multline*}
I_N
%=G^{-\top}_n(x)\left(G^{\top}_{n+1}(x)C_{n+1}+G^{\top}_{n-1}(x)A_{n-1}\right)^{\prime}Q_{n}(x) \\ -G^{-\top}_n(x)(G^{\top}_n(x))^{\prime}G^{-\top}_n(x)\left( G^{\top}_{n+1}(x)C_{n+1}+G^{\top}_{n-1}(x)A_{n-1}\right)
%\\
=G^{-\top}_n(x) \, \big(G^{\top}_{n+1}(x) \, C_{n+1}+G^{\top}_{n-1}(x) \, A_{n-1}\big)^{\prime} \, Q_{n}(x)
 \\ +(G^{-\top}_n(x))^{\prime}\big( G^{\top}_{n+1}(x) \, C_{n+1}+G^{\top}_{n-1}(x) \, A_{n-1}\big) \, ,
\end{multline*}
and so
%\begin{gather*}
$ I_N =\Big(G^{-\top}_n(x) \, \big(G^{\top}_{n+1}(x) \, C_{n+1}+G^{\top}_{n-1}(x) \, A_{n-1}\big)\Big)^{\prime} \, $.
%\end{gather*}
Integrating the above with respect to the variable $x$, we get that $\big\{ G_n \big\}_{n\in\mathbb N} \, $ satisfies the following recurrence relation
\begin{gather}\label{e1}
G^{\top}_n(x) \, (xI-B_n)=G^{\top}_{n+1}(x) \, C_{n+1}+G^{\top}_{n-1}(x) \, A_{n-1} \, .
\end{gather}
A similar procedure for $ \big\{ Q_n \big\}_{n\in\mathbb N} \, $ yield
\begin{gather}\label{e2}
(xI-\widetilde B_n ) \, Q_n (x) = A_n \, Q_{n+1}(x) + C_n \, Q_{n-1} (x) \, .
\end{gather}
From~\eqref{ide},~\eqref{e1} and~\eqref{e2}, we obtain
%\begin{gather*}
$%\displaystyle
\sum_{k=0}^n G^\top(x) \, (\widetilde B_k-B_k) \, Q_n(x)={\bf 0} \, $,
%\end{gather*}
and this implies that $\widetilde B_n=B_n$, for every $n\in\mathbb N \, $.

Since $ \big\{ Q_n \big\}_{n\in\mathbb N} \, $ and $ \big\{ V_n \big\}_{n\in\mathbb N} \, $ satisfy the same recurrence relation (with different initial conditions),
from the Favard's theorem we can conclude, that there exist a matrix of linear functionals such that $ \big\{ V_n \big\}_{n\in\mathbb N} \, $ and $ \big\{G_n \big\}_{n\in\mathbb N} \, $ are biorthogonal.
\end{proof}

\begin{lema}\label{lema1}
For every $n \in \mathbb N \, $,
%\begin{gather*}
$%\displaystyle
Q_{n-1}(x) \, G_{n-1}^{\top}(x)-V_{n-1}(x) \, R_{n-1}^{\top}(x)=\pmb 0
 \, $.
%\end{gather*}
\end{lema}
\begin{proof}
From the definition of second kind functions we get 	
\begin{multline*}
Q_{n-1}(x) \, G^{\top}_{n-1}(x) - V_{n-1}(x) \, R^{\top}_{n-1}(x)
 %\notag\\[4pt]
 \\
=\prodint{ V_{n-1}(y),\frac{G_{n-1}(x)}{x-y}}_u-\prodint{ \frac{V_{n-1}(x)}{x-y},G_{n-1}(y)}_u
 %\notag\\[4pt]
 \\
=\prodint {V_{n-1}(y), \frac{G_{n-1}(x)-G_{n-1}(y)}{x-y}}_u-\prodint{ \frac{V_{n-1}(y)-V_{n-1}(x)}{x-y},G_{n-1}(y)}_u
% \notag\\[4pt]
%&\label{os3}=\pmb 0
 \, ,
\end{multline*}
and from the orthogonality conditions %for $ \{ V_n \} $, $ \{ G_n \} $
the result follows.
\end{proof}
\begin{proposition}[Liouville-Ostrogradski type formulas]
Let $ \big\{ V_n \big\}_{n\in\mathbb N} \, $, %and \linebreak
$\big\{ G_n \big\}_{n\in\mathbb N} \, $ be the sequences of biorthogonal polynomials with respect to a matrix of linear functionals $u$ and let $\big\{ Q_n \big\}_{n\in\mathbb N} \, $ and $\big\{ R_n \big\}_{n\in\mathbb N} \, $ be their respective sequences of second kind functions,~then
\begin{gather}
\label{os1}
Q_{n-1}(x) \, G^{\top}_{n}(x)-V_{n-1}(x) \, R^{\top}_{n}(x)
 = C_{n}^{-1} \, , \\[1.25pt]
\label{os2}
V_{n}(x) \, R^{\top}_{n-1}(x)-Q_{n}(x) \, G^{\top}_{n-1}(x)
 = A_{n-1}^{-1} \, .
\end{gather}
\end{proposition}
\begin{proof}
We will prove~\eqref{os1} and~\eqref{os2} follows by using analogous arguments. We proceed by induction. For $n=0$ the result is obtained from initial conditions. Suppose now that
\begin{gather*}
Q_{k-1}(x) \, G^{\top}_{k}(x)-V_{k-1}(x) \, R^{\top}_{k}(x)=C_{k}^{-1} \, ,
\ k=0,1 ,\ldots , n-1 \, .
\end{gather*}
%for all $k=0,1 ,\ldots , n-1 \, $.
Then, from the recurrence relation
for $G^{\top}_n $ and~$R^{\top}_{n}$
\begin{multline*}
Q_{n-1}(x) \, G^{\top}_{n}(x)-V_{n-1}(x) \, R^{\top}_{n}(x)
= \big(V_{n-1} \, R^{\top}_{n-1}-Q_{n-1} \, G^{\top}_{n-2} \big) \, A_{n-2} \, C_n^{-1}
 \\
+ \big(Q_{n-1}(x) \, G^{\top}_{n-1}-V_{n-1} \, R^{\top}_{n-1}\big) \, (x-B_{n-1}) \, C_n^{-1} \, .
\end{multline*}
Thus from Lemma~\ref{lema1},
\begin{gather*}
%{multline*}
\phantom{o}
\hspace{-.5cm}
Q_{n-1}(x) \, G^{\top}_{n}(x)-V_{n-1}(x) \, R^{\top}_{n}(x)
 %\\
=\big(V_{n-1}(x) \, R^{\top}_{n-1}(x)-Q_{n-1}(x) \, G^{\top}_{n-2}(x)\big) \, A_{n-2} \, C_n^{-1} .
\end{gather*}
%{multline*}
If now we use the recurrence formulas for $V_{n-1} $ and $Q_{n-1} $, then from induction hypothesis and Lemma~\ref{lema1}, we~get
\begin{align*}
 & \Big(V_{n-1}(x) \, R^{\top}_{n-1}(x) -Q_{n-1}(x) \, G^{\top}_{n-2}(x)\Big) \, A_{n-2} \, C_n^{-1} \phantom{olaolaolaolaol} \\
 & \phantom{olaol} = A_{n-2}^{-1} (x-B_{n-2}) \big( V_{n-2}(x) \, R^{\top}_{n-2}(x)-Q_{n-2}(x) \, G^{\top}_{n-2}(x)\big) \, A_{n-2} \, C_n^{-1} \\
 & \phantom{olaolaolaolaol}
 +A_{n-2}^{-1} C_{n-2} \big( Q_{n-3}(x) \, G^{\top}_{n-2}(x)-V_{n-3}(x) \, R^{\top}_{n-2}(x) \big) \, A_{n-2} \, C_n^{-1} \\
% & \phantom{ola} =A_{n-2}^{-1}C_{n-2}C_{n-2}^{-1}A_{n-2}C_n^{-1} \, ,
% %\\ &=C_n^{-1}.
%\end{align*}
 & \phantom{olaol}
= A_{n-2}^{-1} \, C_{n-2} \, \big( Q_{n-3}(x) \, G^{\top}_{n-2}(x)-V_{n-3}(x) \, R^{\top}_{n-2}(x) \big) \, A_{n-2} \, C_n^{-1} \, ,
\end{align*}
and the result follows from the induction hypothesis.
\end{proof}

\section{Casorati determinants.}\label{sectio4}

Consider the matrix second-order recurrence relations
\begin{gather}
 \label{recurrencerelation}
x \, y_n=A_{n} \, y_{n+1}+B_n \, y_n+C_n \, y_{n-1} \, , \ \ n\geq 0 \, ,
 \\
%\end{gather}
%\begin{gather}
\label{recurrencerelation1}
x \, t_n= t_{n+1} \, C_{n+1}+t_n \, B_n+t_{n-1} \, A_{n-1} \, , \ \ n\geq 0 \, .
\end{gather}
\begin{teo}\label{teorema2}
If $\big\{ w_n \big\}_{n\in\mathbb N} \, $ and $\big\{ v_n \big\}_{n\in\mathbb N} \, $ are solutions of~\eqref{recurrencerelation}, then
\begin{gather} \label{eq:nova}
\det \, \big( W(w_n,v_n) \big)=\det \, (A^{-1}_{n}) \, \det \, (C_n) \, \det \, \big( W(w_{n-1},v_{n-1}) \big) \, .
\end{gather}
\end{teo}

\begin{proof}
First of all we recall that
%\begin{gather*}
$W (w_n,v_n) =
\left( \begin{matrix}
w_n & v_n \\ w_{n+1} & v_{n+1}	
 \end{matrix} \right) \, $.
%\end{gather*}
Thus from Schur complement
%we have that
\begin{gather*}
\det \, \big( W(w_n,v_n) \big)= \det \, (w_n) \, \det \,
\big( v_{n+1}-w_{n+1} w_n^{-1}v_n \big) \, .
\end{gather*}
Since
$ \big(w_n \big)_{n\in\mathbb N} \, $
and
$ \big(v_n \big)_{n\in\mathbb N} \, $
are solutions of~\eqref{recurrencerelation}, then
\begin{gather} \label{1}
v_{n+1} = A_{n}^{-1} ( x \, v_n - B_n \, v_n - C_n \, v_{n-1} ) ,
%\, , \\ \ \ \mbox{ and } \ \
w_{n+1} = A_{n}^{-1} ( x \, w_n - B_n \, w_n - C_n \, w_{n-1} ) \, .
%\notag
\end{gather}
Thus
\begin{gather}\label{3}
- w_{n+1} w_n^{-1}v_n=-A_{n}^{-1}
\big( x \, v_n- B_n \, v_n-C_n \, w_{n-1} \, w_n^{-1} \, v_n \big) \, .
\end{gather}
If we subtract~\eqref{3} from the first equation in~\eqref{1}
\begin{gather}\label{deter}
v_{n+1}-w_{n+1}w_n^{-1}v_n
=-A_{n}^{-1} C_n \, \big( v_{n-1} - w_{n-1} \, w_n^{-1} v_n \big) \, .
\end{gather}
As a consequence,
\begin{multline}
 \label{determinante}
\det \, \big( W (w_n,v_n) \big)
 \\
=(-1)^p \det \, (w_n) \, \det \, (A_{n}^{-1}) \, \det \,(C_n) \, \det \,(v_{n-1}-w_{n-1}w_n^{-1}v_n) \, .
\end{multline}
On the other hand, if now we consider the matrix
$%\det \big(
W(w_{n-1},v_{n-1})
%\big)
$,
using the fact~that
\begin{gather*}
\left( \begin{matrix}
w_{n-1}&v_{n-1} \\
w_n&v_n
 \end{matrix} \right)
\left( \begin{matrix}
-w_{n}^{-1} v_n & w^{-1}_n\\
I_N &0
 \end{matrix} \right)=
\left( \begin{matrix}
v_{n-1}-w_{n-1} w_n^{-1}v_n & w_{n-1} w_n^{-1}\\
0&I_N
 \end{matrix} \right) \, , %\\ %[10pt]
\end{gather*}
and
\begin{gather*}
\left( \begin{matrix}
w_{n-1} & v_{n-1} \\
w_n & v_n
 \end{matrix} \right)
\left( \begin{matrix}
I_N &0\\
-v_{n-1}^{-1} w_{n-1} & v_{n-1}^{-1}
 \end{matrix} \right)=
\left( \begin{matrix}
0&I_N \\
w_{n-1}-v_{n}v_{n-1}^{-1} w_{n-1} & v_n v_{n-1}
 \end{matrix} \right) \, ,
\end{gather*}
and Proposition~\ref{detpro}, we have that
\begin{align*}
 & \det \, \big( W(w_{n-1},v_{n-1}) \big)
  =  (-1)^p \det \, (w_n) \, \det \, (v_{n-1}-w_{n-1}w_n^{-1}v_n)
 \\
  & \phantom{\det \, \big( W(w_{n-1},v_{n-1}) \big)}
  \, =  (-1)^p \det \, (v_{n-1}) \, \det \, (w_{n}-v_{n}v_{n-1}^{-1}w_{n-1}) \, .
\end{align*}
Replacing the above in~\eqref{determinante} we~get~\eqref{eq:nova}.
%\begin{gather*}
%\det \, \big( W(w_n,v_n) \big)=\det \, (A^{-1}_{n}) \, \det \, (C_n) \, \det \, \big( W(w_{n-1},v_{n-1}) \big) \, .
%\end{gather*}
\end{proof}
\begin{proposition}\label{proposition5}
The sequences $\big\{ V_n \big\}_{n\in\mathbb N} \, $ and $\big\{ V^{(1)}_{n-1} \big\}_{n\in\mathbb N} \, $ with initial conditions $(V_{-1},V_0)=({\bf 0},I_N )$ and $(V^{(1)}_{-2},V^{(1)}_{-1})=(-A_0,{\bf 0})$ are linearly independent solutions of~\eqref{recurrencerelation}.
Moreover, they constitute a basis of $ \, \mathbb S \, $.
\end{proposition}
\begin{proof}
Since $\big\{ V_n \big\}_{n\in\mathbb N} \, $ and $ \big\{V^{(1)}_{n-1} \big\}_{n\in\mathbb N} \, $ are solutions of~\eqref{recurrencerelation}, then using Theorem~\ref{teorema2},
we successively~get
\begin{align*}
 & \det \, \big( W(V_n,V^{(1)}_{n-1}) \big)
 = \det \, (A^{-1}_{n}) \, \det \, (C_n) \, \det \, \big( W(V_{n-1}(x),V^{(1)}_{n-2}(x))\big)
 \phantom{olaola} \\
%\end{gather*}
%Iterating,
%\begin{eqnarray*}
%\det \, \big( W(V_n,V^{(1)}_{n-1}) \big)
 & \phantom{olaola} =  \prod_{j=1}^{n} \det \, (A^{-1}_{j}) \, \det \, (C_j) \, \det \, \big( W(V_{-1}(x),V^{(1)}_{-2}(x)) \big) \\
 & \phantom{olaola} =  \det \, (A_0) \, \prod_{j=1}^{n} \det \, (A^{-1}_{j}) \, \det \, (C_j) \, .
\end{align*}
Now, as for every $n\in \mathbb N \, $, $\det \, (A_{n}) \ne 0$ and $\det \, (C_{n})\ne 0$, we get that $\big\{ V_n \big\}_{n\in\mathbb N} \, $ and $ \big\{V^{(1)}_{n-1} \big\}_{n\in\mathbb N} \, $ are linearly independent.

Recall that if $ \big(y_{n}(e_i) \big)_{n \in \mathbb N} \, , \, i=0,1\, $ are the solutions of~\eqref{recurrencerelation} with initial conditions $(I_N ,0)$ and $(0,I_N )$, respectively, then
\begin{gather*}
V_n(x)=y_{n}(e_1) \, , \ \ V_{n-1}^{(1)}(x) =-y_n (e_0) \, A_0 \, ,
\end{gather*}
and so
\begin{gather*}
\left( \begin{matrix} V_{n-1}^{(1)}(x) & V_{n}(x) \end{matrix} \right)
=
\left( \begin{matrix} y_{n}(e_0) & y_{n}(e_1) \end{matrix} \right)
\left( \begin{matrix}
-A_0 & 0 \\ 0 & I_N
\end{matrix}
\right) \, ;
\end{gather*}
which implies that $\big\{ V_n \big\}_{n\in\mathbb N} \, $ and $ \big\{V^{(1)}_{n-1} \big\}_{n\in\mathbb N} \, $ constitute a basis for $ \, \mathbb S \, $.
\end{proof}

From Proposition~\ref{proposition5} it follows that every solution of~\eqref{recurrencerelation} is a linear combination
%(to right)
of $ \big\{ V_n \big\}_{n\in\mathbb N} \, $ and $ \big\{ V^{(1)}_{n-1} \big\}_{n\in\mathbb N} \, $. In particular
\begin{gather}\label{cone}
V^{(k)}_{n-k}(x)
=V_n(x) \, \gamma_k
%\mathrm C_k
+V^{(1)}_{n-1}(x) \, \eta_k
%\mathrm D_k
\, .
\end{gather}
Taking $n=k$ and $n=k-1$, we get the representation
\begin{align*}
%\mathrm C_k
\gamma_k = \Theta_*
\left(\begin{matrix}
V^{(1)}_{k-2}(x)&V_{k-1}(x) \\[1.25pt]
V^{(1)}_{k-1}(x)&V_{k}(x)
\end{matrix}\right)^{-1} \, \ \mbox{ and } \ \
%\mathrm D_k
\eta_k = &\Theta_*\left(\begin{matrix}
V_{k-1}(x)&V^{(1)}_{k-2}(x) \\[1.25pt]
V_{k}(x)&V^{(1)}_{k-1}(x)
\end{matrix}\right)^{-1} \, .
\end{align*}
%Notice that from proposition~\ref{proposition5}, $\operatorname{det}\begin{psmallmatrix}P_{k-1}(x)&P^{(1)}_{k-2}(x)\\P_{k}(x)&P^{(1)}_{k-1}(x)\\ \end{psmallmatrix}$ is a constant. (non depend of variable $x$) thus $\begin{psmallmatrix}P_{k-1}(x)&P^{(1)}_{k-2}\\P_{k}(x)&P^{(1)}_{k-1}\\ \end{psmallmatrix}^{-1}$ is also a matrix polynomial. Moreover from~\eqref{} we obtain that $D_k$ is a matrix polynomial.
In particular, if we take $k=2$, from~\eqref{deter} we obtain
\begin{gather*}
x \, V_{n-1}^{(1)}(x) = V_{n}(x) \, A_{0} + V_{n-1}^{(1)}(x) \, A_{0}^{-1} B_{0}\, A_{0} + V^{(1)}_{n-2}(x) \, A_{1}^{-1} C_1 \, A_{0} \, .
\end{gather*}
\begin{proposition}
Let $\big\{ V_n \big\}_{n\in\mathbb N} \, $ and $\big\{ V^{(k)}_{n-1} \big\}_{n\in\mathbb N} \, $
%be defined by
satisfies~\eqref{recurrencerelation}
with initial conditions $(V_{-1},V_0)=({\bf 0},I_N )$ and $(V^{(k)}_{-2},V^{(k)}_{-1})=(-C_{k-1}^{-1} A_{k-1},{\bf 0})$.
Then, %the following identity holds
\begin{multline}\label{buena}
x \, V_{n-1}^{(k)}(x)
 \\
=V_{n}^{(k-1)}(x) \, A_{k-1} + V_{n-1}^{(k)}(x) \, A_{k-1}^{-1} B_{k-1} \, A_{k-1} + V^{(k+1)}_{n-2}(x) \, A_{k}^{-1} C_k \, A_{k-1} \, .
\end{multline}
\end{proposition}
\begin{proof}
From~\eqref{cone} and taking into account~\eqref{deter} we get	
\begin{multline}
V_{n-2}^{(k+1)}(x) \, A^{-1}_{k} \, C_{k} =
-V^{(1)}_{n+k-2}(x)
\, \big( V^{(1)}_{k-2}(x) - V_{k-1}(x) \, (V_{k}(x))^{-1} V_{k-1}^{(1)} \big)^{-1}
\\ -V_{n+k-1}(x) \,
\big(V_{k-1}(x)-V_{k-2}^{(1)}(x) \, (V_{k-1}^{(1)}(x))^{-1} V_k(x) \big)^{-1}
 \, .
\label{125}
\end{multline}
On the other hand, observe that
\begin{align*}
A_{k-1}^{-1}(xI_N -B_{k-1})&= \big(
V_{k-1}^{(1)}(x) + A_{k-1}^{-1} C_{k-1} \, V_{k-3}^{(1)}(x)
\big)
  \,
(V_{k-2}^{(1)}(x))^{-1} \, , \\
A_{k-1}^{-1} (xI_N -B_{k-1}) &=\big(
V_{k}(x)+A_{k-1}^{-1} C_{k-1} \, V_{k-2}(x)
\big)(V_{k-1}(x))^{-1} \, .
\end{align*}
Thus, from the recurrence relation~\eqref{recurrencerelation} and~\eqref{deter}
\begin{multline*}
\big(
V_k(x)-V^{(1)}_{k-1}(x) \, (V_{k-2}^{(1)}(x))^{-1} V_{k-1}(x)
\big)^{-1}
A_{k-1}^{-1} (xI_N -B_{k-1})
%\\
%=\left[V_k(x)-V^{(1)}_{k-1}(x)(V_{k-2}^{(1)}(x))^{-1}V_{k-1}(x)\right]^{-1}\left(V_{k-1}^{(1)}(x)(V^{(1)}_{k-2}(x))^{-1}+A_{k-1}^{-1}C_{k-1} V^{(1)}_{k-3}(x)(V^{(1)}_{k-2})^{-1}\right)\notag \\
 \\
=\big(
V_{k-2}^{(1)}(x) \, (V^{(1)}_{k-3})^{-1} C_{k-1}^{-1} A_{k-1} \, V_{k}(x)
+
\big(
V^{(1)}_{k-2}(x) \, (V^{(1)}_{k-1}(x))^{-1} V_k(x) - V_{k-1}(x)
\big)^{-1}
 \\
-V^{(1)}_{k-2}(x) \, (V_{k-3}^{(1)}(x))^{-1}C_{k-1}^{-1} A_{k-1} \, V^{(1)}_{k-1}(x) \, (V^{(1)}_{k-2}(x))^{-1} V_{k-1}(x)
\big)^{-1} \, ,%\\
%+ \big( V^{(1)}_{k-2}(x) \, (V^{(1)}_{k-1}(x))^{-1} V_k(x) - V_{k-1}(x) \big)^{-1}
\end{multline*}
and so
\begin{multline}
\big(
V_k(x)-V^{(1)}_{k-1}(x)
\, (V_{k-2}^{(1)}(x))^{-1}V_{k-1}(x)
\big)^{-1}
A_{k-1}^{-1} (xI_N -B_{k-1})
\\
=-\big(V_{k-1}(x)-V^{(1)}_{k-2}(x) \, (V^{(1)}_{k-1}(x))^{-1} V_k(x)\big)^{-1}
\\
+ \big(
V_{k-1}(x)-V_{k-2}^{(1)}(x) \, (V_{k-3}^{(1)}(x))^{-1} V_{k-2}(x)
\big)^{-1} \, .
\label{123}
\end{multline}
In the same way, we obtain
\begin{multline}
\big(
V_{k-1}^{(1)}(x)-V_k(x) \, (V_{k-1}(x))^{-1}V^{(1)}_{k-2}(x)
\big)^{-1} A_{k-1}^{-1} (xI_N -B_{k-1})
 \\
=- \big(
V^{(1)}_{k-2}(x)-V_{k-1}(x) \, (V_{k}(x))^{-1}V^{(1)}_{k-1}(x)
    \big)^{-1}
\\
 +
\big(V^{(1)}_{k-2}(x)-V_{k-1}(x) \, (V_{k-2}(x))^{-1}V^{(1)}_{k-3}(x)
\big)^{-1} \, .
\label{124}
\end{multline}
Replacing~\eqref{123} and~\eqref{124} in~\eqref{125}
\begin{multline*}
V_{n-2}^{(k+1)}(x) \, A_{k}^{-1}C_{k}=
%&=V_{n+k-1}(x)\Big(\left[V_k(x)-V^{(1)}_{k-1}(x)(V_{k-2}^{(1)}(x))^{-1}V_{k-1}(x)\right]^{-1}A_{k-1}^{-1}(xI_N -B_{k-1}) \\
%& \ \ \ \ \ \ \ \ \ \ \ \ \ \ \ \ \ \ \ \ \ \ \ \- \left[V_{k-1}(x)-V_{k-2}^{(1)}(x)(V_{k-3}^{(1)}(x))^{-1}V_{k-2}(x)\right]^{-1}\Big) \\
%&+V^{(1)}_{n+k-2}(x)\Big(\left[V_{k-1}^{(1)}(x)-V_k(x)(V_{k-1}(x))^{-1}V^{(1)}_{k-2}(x)\right]^{-1}A_{k-1}^{-1}(xI_N -B_{k-1}) \\ &
%\ \ \ \ \ \ \ \ \ \ \ \ \ \ \ \ \ \ \ \ \ \ \ \-\left[V^{(1)}_{k-2}(x)-V_{k-1}(x)(V_{k-2}(x))^{-1}V^{(1)}_{k-3}(x)\right]^{-1}\Big) \\
\big(V_{n+k-1}(x) \, \gamma_k
%\mathrm C_{k}
+V^{(1)}_{n+k-2}(x) \, \eta_k
%\mathrm D_{k}
\big)
\, A_{k-1}^{-1} (xI_N -B_{k-1}) \\
-\big(
V_{n+k-1}(x) \, \gamma_{k-1}
%\mathrm C_{k-1}
 + V^{(1)}_{n+k-2}(x) \, \eta_{k-1}
% \mathrm D_{k-1}
\big)
\end{multline*}
and so
\begin{gather*}
%\displaystyle
V_{n-2}^{(k+1)}(x) \, A_{k}^{-1} C_{k}
= V_{n-1}^{(k)}(x) \, A^{-1}_{k-1} (xI_N -B_{k-1})-V^{(k-1)}_{n}(x) \, ,
\end{gather*}
and the result follows.
\end{proof}

In a similar way, we~get the following result.
\begin{proposition}
The sequences of matrix polynomials, $ \big\{ G_n \big\}_{n\in\mathbb N} \, $
 and
$ \big\{ G^{(1)}_{n-1} \big\}_{n\in\mathbb N} \, $ are linearity independent solutions of~\eqref{recurrencerelation1}.
 Moreover, for every $k\in\mathbb N \, $,
\begin{gather*}
G^{(k)\top}_{n-k}(x)=
%\mathrm A_k
\widetilde{\gamma}_k \, G^\top_n(x) +
%\mathrm B_k
\widetilde{\eta}_k \, G^{(1)\top}_{n-1}(x) \, ,
\end{gather*}
where
\begin{gather*}
 %\mathrm A_k
\widetilde{\gamma}_k =\Theta_*
\left(\begin{matrix}
G^{(1)\top}_{k-2}(x)&G^{(1)\top}_{k-1}(x) \\[1.25pt]
G^\top_{k-1}(x)&G^\top_{k}(x)
\end{matrix}\right)^{-1} \, \
\mbox{ and } \
 \
%\mathrm B_k
\widetilde{\eta}_k = \Theta_*
\left(\begin{matrix}
G^\top_{k-1}(x)&G^\top_{k}(x) \\[1.25pt]
G^{(1)\top}_{k-2}(x)&G^{(1)\top}_{k-1}(x)
\end{matrix}\right)^{-1} \, ;
\end{gather*}
moreover, the following relation holds
\begin{gather*}%\label{buena1}
x \, G^{(k)\top}_{n-1}(x)=C_kG^{(k-1)\top}_{n}(x)+C_kB_{k-1}C_k^{-1}G^{(k)\top}_{n-1}(x)+C_{k} \,A_{k-1}C^{-1}_{k+1}G_{n-2}^{(k+1)\top}(x) \, .
\end{gather*}
\end{proposition}
Since $ \big\{ Q_n \big\}_{n\in\mathbb N} \, $ and $\big\{ R_n \big\}_{n\in\mathbb N} \, $ are also solutions of~\eqref{recurrencerelation} with initial conditions
$%\displaystyle
Q_0(y) = \prodint{I_N ,\frac{I_N }{y-x} }_{u} $,
$%\displaystyle
Q_{-1}(y)=C_0^{-1} $,
$%\displaystyle
R^{\top}_0(y)=\prodint{\frac{I_N }{y-x},I_N}_{u} \, $, and
$%\displaystyle
R^{\top}_{-1}(y)=A^{-1}_{-1} \, $;
then for
$x\in\mathbb C \setminus \operatorname{supp} \, (u)$ it is clear that
\begin{gather}\label{aso}
Q_{n}(x)=V_n(x) \, Q_0(x)-V^{(1)}_{n-1}(x) \, A_0^{-1} \, , \\R_n^\top(x)=R^\top_0(x) \, G_n^\top(x) - C_1^{-1} G^{(1)\top}_{n-1}(x) \, .
\label{aso1}
\end{gather}
From here we~get
\begin{gather*}
V^{(1)}_{n-1}(x)=\prodint{\frac{V_n(x)-V_n(y)}{x-y},I_N }_u A_0, \ \ G^{(1)\top}_{n-1}(x)=C_1\prodint{ I_N ,\frac{G_n(x)-G_n(y)}{x-y}}_u \, .
\end{gather*}
Using a similar argument as in Proposition~\ref{proposition5},
\begin{gather*}
\det \, (W(Q_n,V^{(k)}_{n-k}))=\prod^n_{j=k} \det \, (A^{-1}_{j}) \, \det \, (C_j) \, \det \, (Q_{k-1}(x)) \, ,
\end{gather*}
\begin{proposition}\label{prop10}
The sequences $\big\{ Q_n \big\}_{n\in\mathbb N} \, $ and $ \big\{V^{(k)}_{n-k} \big\}_{n\in\mathbb N} \, $
%also
are linearly independent solutions of~\eqref{recurrencerelation} in $x \in \mathbb C \setminus \operatorname{supp} \, (u)$.
Moreover,
%From here
\begin{gather*}
%\displaystyle
V^{(k)}_{n-k}(x)=V_n(x) \, %\mathrm
{\alpha}_k
- Q_{n}(x) \, %\mathrm
{\beta}_k \, ,
\end{gather*}
with
\begin{gather*}
 \alpha_k=\Theta_*\left(\begin{matrix}
Q_{k-1}(x)&V_{k-1}(x) \\[1.25pt]
Q_{k}(x)&V_{k}(x)
\end{matrix}\right)^{-1} \, \ \mbox{ and } \ \ \beta_k=\Theta_*\left(\begin{matrix}
V_{k-1}(x)&Q_{k-1}(x) \\[1.25pt]
V_{k}(x)&Q_{k}(x)
\end{matrix}\right)^{-1} \, .
\end{gather*}
In the same way, $ \big\{ R^\top_n \big\}_{n\in\mathbb N} \, $
and
$\big\{ G^{(k)\top}_{n-k} \big\}_{n\in\mathbb N} \, $ also are linearly independent solutions of~\eqref{recurrencerelation1} in $x\in\mathbb C \setminus\operatorname{supp} \, (u)$; moreover,
\begin{gather*}
(G^{(k)}_{n-k}(x))^\top=\widetilde{\alpha}_k \, G^{\top}_n(x)
- \widetilde{\beta}_k \, R^\top_{n}(x) \, ,
\end{gather*}
with
\begin{gather*}
\widetilde\alpha_k=
\Theta_*\left(\begin{matrix}
R^\top _{k-1}(x)&R^\top _{k}(x) \\[1.25pt]
G^\top _{k-1}(x)&G^\top _{k}(x)
\end{matrix}\right)^{-1} \, \ \mbox{ and } \ \
 \widetilde\beta_k=\Theta_*\left(\begin{matrix}
G^\top _{k-1}(x)&G^\top _{k}(x) \\[1.25pt]
R^\top _{k-1}(x)&R^\top _{k}(x)
\end{matrix}\right)^{-1} \, .
\end{gather*}
\end{proposition}
\begin{proposition}\label{cris1}
The following Christoffel-Darboux formulas hold
\begin{gather*} 	
\sum_{k=1}^{n}V^{(k)}_{n-k}(y) \, A_{k-1}^{-1} \, V_{k-1}(x)=\frac{V_n(x)-V_n(y)}{x-y} \, , \\
\sum_{k=1}^{n}G^\top_{k-1}(x) \, C_k^{-1} \, G^{(k)\top}_{n-k}(y)
 =\frac{G^\top_n(x)-G^\top_n(y)}{x-y} \, ,
\end{gather*}
as well as its confluent expression
\begin{gather*} 	
\sum_{k=1}^{n}V^{(k)}_{n-k}(x) \, A_{k-1}^{-1} V_{k-1}(x)=V^\prime_n(x) \, ,
 \\ 	
\sum_{k=1}^{n}G^\top_{k-1}(x) \, C_k^{-1} G^{(k)\top}_{n-k}(x)=(G^\top_n(x))^\prime \, .
\end{gather*}
\end{proposition}

\begin{proof}
For $k\leq n$
\begin{gather*}
y \, V_{n-k}^{(k)}(y) \, A_{k-1}^{-1}
 =
V_{n-k+1}^{(k-1)}(y)+V_{n+k}^{(k)}(y) \, A_{k-1}^{-1} \, B_{k-1}+V^{(k+1)}_{n-k-1}(y) \, A_{k}^{-1} \, C_k \, , \\
x \, A_{k-1}^{-1} V_{k-1}(x)
 =
V_k(x) + A_{k-1}^{-1} B_{k-1} \, V_{k-1}(x)+A_{k-1}^{-1} C_{k-1} \, V_{k-2}(x) \, .
\end{gather*}
From here
\begin{multline*}
\phantom{o}
\hspace{-.75cm}
(y-x) \, V_{n-k}^{(k)}(y) \, A_{k-1}^{-1} V_{k-1}(x) = \Big( V_{n-k+1}^{(k-1)}(y) \, V_{k-1}(x)
-V_{n-k}^{(k)}(y) \, A_{k-1} \, C_{k-1} \, V_{k-2}(x)\Big) \\
-\big(
V_{n-k}^{(k)}(y) \, V_{k}(x)-V_{n-k-1}^{(k+1)}(y) \, A_{k} \, C_{k} \, V_{k-1}(x)
\big) \, .
\end{multline*}
Summing the above on $k$ from $1$ to $n$ and taking into account that for every $k\in \mathbb N \, $, $V^{(k)}_{0}(x)=I_N $ and $V^{(k)}_{-1}(x)=\pmb 0 $ we get the result. The confluent form is obtained when $y$ tends to $x$. Formulas for the sequences $ \big\{ G_n \big\}_{n\in\mathbb N} \, $
and
$\big\{ G_n^{(1)} \big\}_{n\in\mathbb N} \, $ are deduced in a similar way.
\end{proof}

As a consequence of Proposition~\ref{cris1}, we find that for $k=1, \ldots ,n $,
\begin{gather*}
V^{(k)}_{n-k}(x) \, A^{-1}_{k-1}
=\sum_{j=1}^{n}V^{(j)}_{n-j}(x) \, A_{j-1}^{-1}\prodint{V_{j-1}(y),G_{k-1}(y)}_u  =\prodint{\frac{V_n(y)-V_n(x)}{y-x},G_{k-1}}_u \, , \\
%\end{multline*}
%\begin{multline*}
C_k^{-1}G^{(k)\top}_{n-k}(x)
=\sum_{j=1}^{n}\prodint {V_{k-1}(y), G_{j-1}(y)}_u C_j^{-1}G^{(j)\top}_{n-j}(x)
\phantom{olaolaolaolaolaolaolaol}
   \\
\phantom{olaolaolaolaolaolaolaolaolaol} =\prodint{ V_{k-1}(y),\frac{G_n(y)-G_n(x)}{y-x}}_u \, .
\end{gather*}
When $k=1$, we recover the classical formula for $V^{(1)}_{n-1}(x)$, %and
$G^{(1)}_{n-1}(x)$.
Algebraic manipulations for the above equations yield
\begin{gather*}
V^{(k)}_{n-k}(x) \, A_{k-1}^{-1}=V_n(x) \, R^\top_{k-1}(x)-Q_{n}(x) \, G^\top_{k-1}(x) \, ,
 \\[1.25pt]
%\end{gather*}
%\begin{gather*}
C_{k}^{-1} G^{(k)\top}_{n-k}(x)=Q_{k-1}(x) \, G^\top_n(x)-V_{k-1}(x) \, R^\top_{n}(x) \, .
\end{gather*}
The above, together with Proposition~\ref{prop10}, yield
\begin{gather*}
%\displaystyle
\phantom{o}
\hspace{-.5cm}
R_{k-1}^\top \, A_{k-1}= \Theta_*\left(
\begin{matrix}
Q_{k-1}(x)&V_{k-1}(x) \\[1.25pt]
Q_{k}(x)&V_{k}(x)
\end{matrix}\right)^{-1} ,
%\displaystyle
G^\top_{k-1}(x) \, A_{k-1}=\Theta_*
\left(
\begin{matrix}
V_{k-1}(x)&Q_{k-1}(x) \\[1.25pt]
V_{k}(x)&Q_{k}(x)
\end{matrix}
\right)^{-1} , \\
C_k \, Q_{k-1}(x)= \Theta_*
\left(\begin{matrix}
R^\top _{k-1}(x)&R^\top _{k}(x) \\[1.25pt]
G^\top _{k-1}(x)&G^\top _{k}(x)
\end{matrix}
\right)^{-1} ,
%\displaystyle
C_k \, V_{k-1}(x)=\Theta_*
\left( \begin{matrix}
G^\top _{k-1}(x)&G^\top _{k}(x) \\[1.25pt]
R^\top _{k-1}(x)&R^\top _{k}(x)
\end{matrix} \right)^{-1} .
\end{gather*}

\section{Outer Ratio Asymptotics}\label{sectio5}

First of all, we are going to state two important theorems which can be find in~\cite{AB_Paco_AM2,AB_Paco_AM1}, see also~\cite{Duran1,Duran4}.
\begin{teo}[Markov type Theorem]
Let $\big\{ V_n \big\}_{n\in\mathbb N} \, $ and $\big\{ G_n \big\}_{n\in\mathbb N} \, $ be the sequence of matrix biorthogonal polynomials with respect to $u$ and $\big\{ V^{(1)}_n \big\}_{n\in\mathbb N} \, $ and $\big\{ G^{(1)}_n \big\}_{n\in\mathbb N} \, $ be the corresponding sequences of associated polynomials. Then,
\begin{gather*}
\lim_{n \to \infty} V_{n}^{-1}(x) \, V_{n-1}^{(1)}(x)
 = \prodint {\frac{I_N }{x-y},I_N }_u \, A_0 \, , \ \ x \in \mathbb C \setminus\Gamma^{(0)} \\
 %\end{gather*}
 %\begin{gather*}
\lim_{n \to \infty} G_{n-1}^{(1)\top}(x) \, G_{n}^{-\top}(x)
 = C_1 \, \prodint {\frac{I_N }{x-y},I_N }_u \, , \ \ x \in \mathbb C \setminus\Gamma^{(0)}
\end{gather*}
and the convergence holds uniformly on compact subsets of $\mathbb C \setminus\Gamma^{(0)} $.
\end{teo}

In the sequel, given three matrices $\alpha,\gamma$ (nonsingular) and $\beta $, we define the {\it left-orthogonal second kind Chebyshev matrix polynomials},
$\big\{ U_n^{\gamma,\beta,\alpha } \big\}_{n\in\mathbb{N}}$, by the recurrence formula
\begin{gather}\label{Chebyshev_recurrence_relation}
x \, U_n^{\gamma,\beta,\alpha }(x)=\gamma \, U_{n+1}^{\gamma,\beta,\alpha }(x)+\beta \, U_n^{\gamma,\beta,\alpha }(x) + \alpha \, U_{n-1}^{\gamma,\beta,\alpha}(x), \ \ n\geq 0 \, ,
\end{gather}
with initial conditions
$U_0^{\gamma,\beta,\alpha }(x)=I_N $
and
$U_{-1}^{\gamma,\beta,\alpha }(x)=\pmb 0 $,
as well as
the {\it right-orthogonal second kind Chebyshev matrix polynomial}, $ \big\{ T_{m}^{\alpha ,\beta ,\gamma}(x) \big\}_{m\in\mathbb{N}}$, %\linebreak
given~by
\begin{gather*}\label{right_chebyshev_matrix_polynomials}
x \, T_{n}^{\alpha ,\beta ,\gamma}(x)=T_{n+1}^{\alpha ,\beta ,\gamma}(x) \, \alpha +T_{n}^{\alpha ,\beta ,\gamma}(x) \, \beta +T_{n-1}^{\alpha ,\beta ,\gamma}(x) \, \gamma \, , \ \ n\geq 0 \, ,
\end{gather*}
with initial conditions
$T_0^{\alpha ,\beta ,\gamma}(x)=I_N $
and
$T_{-1}^{\alpha ,\beta ,\gamma}(x)=\pmb 0 $.
We denote by $u^{\gamma,\beta ,\alpha }$ the matrix of linear functionals for which the polynomials $U_n^{\gamma,\beta ,\alpha }(x)$
and
$T_{n}^{\alpha,\beta,\gamma}(x)$ are biorthogonal.
\begin{teo} [Outer Ratio Asymptotics] \label{teo4}
Let $ \big\{ V_n \big\}_{n\in\mathbb N} \, $ and $\big\{ G_n \big\}_{n\in\mathbb N} \, $
 be the sequences of matrix biorthonormal polynomials with respect to $u \, $. If $A_n\to A$, $B_n\to B$, and $C_n\to C$ with~$A, C$ nonsingular matrices, then
\begin{gather*}
\lim_{n\to\infty}V_{n-1}(x) \, V^{-1}_n(x) \, A_{n-1}^{-1}=\prodint {\frac{I_N }{x-y},I_N }_{u^{C,B,A}} \, , \ \ x\in\mathbb C \setminus \Gamma^{(0)} \, , \\
%\end{gather*}
%\begin{gather*}
\lim_{n\to\infty}C_{n}^{-1} \, G^{-\top}_{n}(x) \, G^\top_{n-1}(x)=\prodint {\frac{I_N }{x-y},I_N }_{u^{C,B,A}} \, , \ \ x\in\mathbb C \setminus \Gamma^{(0)} \, ;
\end{gather*}
and the convergence holds uniformly on compact subsets of $\mathbb C \setminus\Gamma^{(0)} $.
Moreover, if
%\begin{gather*}
$%\displaystyle
F_{C,B,A}(x)=\prodint {\frac{I_N }{x-y},I_N }_{u^{C,B,A}} \, $,
%\end{gather*}
then $F_{C,B,A}(x)$ is an analytic matrix function satisfying the matrix equation
\begin{gather*}
C \, F_{C,B,A}(x) \, A \, F_{C,B,A}(x)+(B-xI_N ) \, F_{C,B,A}(x)+I_N ={\bf 0} \, .
\end{gather*}
\end{teo}

\begin{cor}
With the conditions of Theorem~\ref{teo4},
for $x\in\mathbb C \setminus\Gamma^{(0)}$, the following limits hold
 uniformly on compact subsets of $\mathbb C \setminus\Gamma^{(0)} $:
\begin{gather} 	
 \lim_{n\to\infty}V^{-1}_{n}(x) \, Q_{n}(x)=\pmb 0 \, , \ \ \lim_{n\to\infty}G^{-1}_{n}(x) \, R_{n}(x)=\pmb 0 \, , \label{12} \\
%\end{gather}
%\begin{gather}
\lim_{n\to \infty} V_{n-1}^{-1}(x)C^{-1}_{n}G_n^{-\top}(x)={\bf 0} \, , \ \
\lim_{n\to \infty}
 V_{n}^{-1} \, A^{-1}_{n-1} \, G_{n-1}^{-\top}(x)={\bf 0}\label{13} \, .
\end{gather}
\end{cor}
\begin{proof}
The limits in~\eqref{12} are obtained from~\eqref{aso}-\eqref{aso1} and Markov theorem. On the other hand, from Liouville-Ostrogradski formulas we~get
\begin{align*}
V^{-1}_{n-1}(x) \, Q_{n-1}(x)-R_n^{\top}(x) \, G_n^{-\top}(x)=V_{n-1}^{-1}(x) \, C_n^{-1} \, G_n^{-\top} \, , \\
R^{\top}_{n-1}(x) \, G_{n-1}^\top(x)-V_{n}(x) \, Q_n(x)=V_{n}^{-1}(x) \, A_{n-1}^{-1} \, G_{n-1}^{-\top} \, .
\end{align*}
Taking the limit when $n\to\infty$ in the above identities we get~\eqref{13}.
\end{proof}

\begin{teo}
Let assume the moment problem for $u$ is determined and $\big\{ V_n \big\}_{n\in\mathbb N} \, $ and $\big\{ G_n \big\}_{n\in\mathbb N} \, $ be the sequence of matrix biorthogonal polynomials with respect to $u$, as well as $\big\{ V^{(k)}_n \big\}_{n\in\mathbb N} \, $ and $\big\{ G^{(k)}_n \big\}_{n\in\mathbb N} \, $ be the respective sequences of the $k$-th associated polynomials. Then, for all $k\in\mathbb N \, $, the following~limits
\begin{gather*}
%\displaystyle
\lim_{n\to\infty}V^{-1}_n(x) \, V^{(k)}_{n-k}(x)=R^\top_{k-1}(x) \, A_{k-1} \, ,
\ \ x \in \mathbb C \setminus\Gamma^{(k)} \\
%\displaystyle
\lim_{n\to\infty}G^{(k)\top}_{n-k}(x) \, G^{-\top}_n(x)=C_k \, Q_{k-1}(x) \, ,
\ \ x \in \mathbb C \setminus\Gamma^{(k)}
\end{gather*}
holds uniformly on compact subsets of $\mathbb C \setminus\Gamma^{(k)} $.
\end{teo}

\begin{proof}
We only prove the first formula. The second one follows in a similar way. We use induction on $k$. When $k=1$ the result is a straightforward consequence of the matrix Markov theorem. Now assume that the result holds true for a~$k$, by~\eqref{buena}
we have that
\begin{gather*}
V^{(k+1)}_{n-k-1}(x)=
\big(x \, V_{n-k}^{(k)}(x) \, A^{-1}_{k-1} - V_{n-k+1}^{(k-1)}(x) - V_{n-k}^{(k)}(x) \, A_{k-1}^{-1} \, B_{k-1}
\big) \, C^{-1}_k \, A_{k} \, ,
\end{gather*}
so, from induction hypothesis
\begin{gather*}
\lim_{n\to\infty}
V^{-1}_n(x) \, V^{(k+1)}_{n-k-1}(x)
=(x \, R^{\top}_{k-1}(x)-R^{\top}_{k-2}(x) \, A_{k-2}
-Q^{\top}_{k-1}(x) \, B_{k-1}) \, C_{k}^{-1} \, A_k \, ,
% \\ &=R^\top_{k}(x) \, A_{k} \, ,
\end{gather*}
and so the convergence of $\big\{ V^{-1}_n(x) \, V^{(k+1)}_{n-k-1}(x) \big\} $ holds uniformly on compact subsets of $\mathbb C \setminus\Gamma^{(k)} $ to $R^\top_{k}(x) \, A_{k}$.
\end{proof}

Let $u^{(k)}$ be the matrix of linear functionals and $\big\{ V^{(k)}_n \big\}_{n \in \mathbb N} \, $, $ \big\{ G^{(k)}_n \big\}_{n \in \mathbb N} \, $ the corresponding matrix biorthogonal polynomials with respect to $u$. From Markov theorem
\begin{align*}
\lim_{n\to\infty}(V^{(k)}_{n}(x))^{-1} V^{(k+1)}_{n-1}(x) &=\prodint{\frac{I_N }{x-y},I_N }_{u^{(k)}} A_{k} \, , \ \ x\in\mathbb C \setminus\Gamma^{(k)} \, , \\
\lim_{n\to\infty}G^{(k+1)\top}_{n-1}(x) \, (G^{(k)}_{n}(x))^{-\top}&= C_{k+1} \, \prodint{\frac{I_N }{x-y},I_N }_{u^{(k)}} \, , \ \ x\in\mathbb C \setminus\Gamma^{(k)} \, ,
\end{align*}
 and the fact that
\begin{gather*}
(V_{n}^{(k)}(x))^{-1} V_{n-1}^{(k+1)}(x)=(V_{n}^{(k)}(x))^{-1} V_{n+k}(x) \, V_{n+k}^{-1}(x) \, V_{n-1}^{(k+1)}(x) \, , \\
%\end{gather*}	
%\begin{gather*}
G_{n-1}^{(k+1)\top}(x) \, (G_{n}^{(k)}(x))^{-\top}= G_{n-1}^{(k+1)\top}(x) \, G_{n+k}^{-\top}(x) \, G^\top_{n+k}(x) \, (G_{n}^{(k)}(x))^{-\top} \, ,
\end{gather*}
we obtain that,
for every $x\in\mathbb C \setminus\Gamma^{(k)}$,
\begin{align}\label{limit}
&\prodint{\frac{I_N }{x-y},I_N }_{u^{(k)}}= A^{-1}_{k-1} \, R_{k-1}^{-\top}(x) \, R^\top_k(x) \, , \\ &\notag\prodint{\frac{I_N }{x-y},I_{p}}_{u^{(k)}}=Q_k(x) \, Q_{k-1}^{-1}(x) \, C^{-1}_{k} \, .
\end{align}
%for every $x\in\mathbb C \setminus\Gamma^{(k)}$.

\begin{teo} \label{teorema 6}
Let $\big\{ V^{(k)}_{n} \big\}_{n\in\mathbb N} \, $ and $ \big\{ G^{(k)}_{n} \big\}_{n\in\mathbb N} \, $ be the sequences of $k$-th associated polynomials which are biorthonormal with respect to the matrix of  linear functionals $u^{(k)}$. If $A_n\to A$, $B_n\to B$, and $C_n\to C$ with $A, C$ nonsingular matrices,~then
\begin{gather*}
\lim_{k\to\infty}\prodint{ \frac{I_N }{z-x},I_{p}}_{u^{(k)}}=\prodint {\frac{I_N }{z-x},I_N }_{u^{A,B,C}}, \ \ z\in\mathbb C \setminus\bigcup_{k}\Gamma^{(k)} \, ,
\end{gather*}
and the convergence is uniform %for $z$
on compact subsets of
$%z\in
\mathbb C \setminus\bigcup_{k}\Gamma^{(k)} $. 	
\end{teo}
\begin{proof}
Firs of all, we will prove that	
\begin{gather}\label{orto}
\lim_{k\to \infty} \prodint{ U^{A,B,C}_{\ell}(x),I_N }_{u^{(k)}}=
I_N \, \delta_{\ell,0}  \, ,
\end{gather}
Since $\big\{ V^{(k)}_n \big\}_{n\in \mathbb N} \, $ is a basis of the bimodule of matrix polynomials, then for each $\ell\in \mathbb N \, $ there exists a set of %numerical
matrices $ \big(\Delta_{i,\ell,k} \big)_{i=0}^l$ such~that
\begin{gather*}
U^{A,B,C}_\ell(x) =\sum_{j=0}^\ell \Delta_{j,\ell,k}V_j^{(k)}(x) \, .
\end{gather*}
From orthogonality
\begin{gather}\label{n0}
\prodint{ U^{A,B,C}_\ell(x),I_N }_{u^{(k)}}=\sum_{j=0}^\ell \Delta_{j,\ell,k}\prodint {V_j^{(k)}(x),I_N }_{u^{(k)}}=\Delta_{0,\ell,k} \, .
\end{gather}
Observe that~\eqref{orto} is proved if
%\begin{gather*}
$ \lim_{k\to \infty} \Delta_{j,\ell,k}=
I_N \, \delta_{\ell,j}  \, $.
%\end{gather*}
We proceed by induction on $\ell $. For $\ell=0$ the result is immediate since from~\eqref{n0}, $\Delta_{0,0,k}=I_N $, and for $j\neq 0$, $\Delta_{j,0,k}={\bf 0} $. Now suppose the result is valid up to $\ell \, $. From~\eqref{Chebyshev_recurrence_relation} we~get
\begin{align*}
\Delta_{j,\ell+1,k}&=\prodint{U^{A,B,C}_{\ell+1}(x),G^{(k)}_{j}(x)}_{u^{(k)}} \\
&=\prodint{A^{-1}
\big(
x \, U^{A,B,C}_\ell(x)-B \, U^{A,B,C}_{\ell}(x)- C \, U^{A,B,C}_{\ell-1}(x)
\big), G^{(k)}_{j}(x)}_{u^{(k)}} \\
&=A^{-1} \prodint{x \, U^{A,B,C}_{\ell}(x),G^{(k)}_{j}(x)}_{u^{(k)}}-A^{-1}B\Delta_{j,\ell,k}-A^{-1}C\Delta_{j,\ell-1,k} \, .
\end{align*}
On the other hand, from %property of
the symmetry condition
\begin{align*}
 & A^{-1} \prodint{U^{A,B,C}_{\ell}(x),x \, G^{(k)}_{j}(x)}_{u^{(k)}}
  \\
  & \phantom{olaola} =
 A^{-1} \prodint{U^{A,B,C}_{\ell}(x),C_{j+k+1}^\top G^{(k)}_{j+1}(x) + B_{j+k}^\top G^{(k)}_j (x) +A_{j+k-1}^\top G^{(k)}_{j-1} (x)}_{u^{(k)}} \\
  & \phantom{olaola} =
 A^{-1}
\big(
\Delta_{j+1,\ell,k} \, C_{k+j+1} + \Delta_{j,\ell,k} \, B_{k+j}+\Delta_{j-1,\ell,k}\, A_{k+j-1}
\big) \, .
\end{align*}
From here we get that
\begin{multline}\label{delta}
\Delta_{j,\ell+1,k}
 = A^{-1}
\big(\Delta_{j+1,\ell,k} \, C_{k+j+1}
+\Delta_{j,\ell,k} \, B_{k+j}+\Delta_{j-1,\ell,k} \, A_{k+j-1}
 \\
-B \, \Delta_{j,\ell,k}-C \, \Delta_{j,\ell-1,k}
\big) \, .
\end{multline}
Observe that for $j\leq \ell-2$ or $j\geq \ell+2$ the induction hypothesis and~\eqref{delta} show that $\lim_{k\to\infty} \Delta_{j,\ell+1,k}={\bf 0}$. Now, for $j=\ell-1$, $\ell$ and $\ell+1$ we~get
\begin{gather*}
\lim_{k\to \infty} \Delta_{\ell-1,\ell+1,k}=A^{-1}C-A^{-1}C={\bf 0} \, , \\
\lim_{k\to \infty} \Delta_{\ell,\ell+1,k}=A^{-1}B-A^{-1}B={\bf 0} \, , \\
%\end{align*}
%\begin{align*}
\lim_{k\to \infty} \Delta_{\ell+1,\ell+1,k}=A^{-1}A=I_N \, .
\end{gather*}
We are now ready to prove that
\begin{gather*}
\lim_{k\to\infty}\prodint{\frac{I_N }{z-x},I_N }_{u^{(k)}}=\prodint{ \frac{I_N }{z-x},I_N }_{u^{A,B,C}} \, , \ \ z\in\mathbb C \setminus \bigcup_{k}\Gamma^{(k)} \, .
\end{gather*}
If the above is not true, then there exist an $z\in\mathbb C \setminus \bigcup_{k}\Gamma^{(k)}$ and a sequence of nonnegative integers~$\big(k_m \big)_{m\in \mathbb N} \, $, such that
\begin{gather}\label{ine}
\left\|\prodint{\frac{I_N }{z-x},I_N }_{u^{(k_m)}}-\prodint{ \frac{I_N }{z-x},I_N }_{u^{A,B,C}}\right\|_2>C>0 \, ,
\end{gather}
where $C$ is a constant. Since $ \big\{ u^{(k)} \big\}_{k \in \mathbb N} \, $ is a sequence of matrices of linear functionals with compact support contained in $\bigcup_{k}\Gamma^{(k)}$ and such that $\prodint{I_N ,I_N }_{u^{(k)}}=I_N $, then from Banach-Alaoglu's theorem, there is a subsequence
$ \big(s_m \big)_{m \in \mathbb N} \, $
from $ \big(k_m \big)_{m \in \mathbb N} \, $
such that $du^{(s_m)}$ converge to a matrix of  linear functionals $v$ with compact support contained in $\pmb{\operatorname{D}}_M$, for every matrix polynomial~$f$,~i.e.
\begin{gather*}
\lim_{m\to\infty}\prodint{f,I_N }_{u^{(s_m)}}=\prodint{f,I_N }_v \, .
\end{gather*}
In particular, if we take $f= U^{A,B,C}_\ell(x)$, then
%\begin{gather*}
$\prodint{ U^{A,B,C}_{\ell}(x),I_N }_{v}= I_N \, \delta_{\ell,0} \, $.
%\end{gather*}
Since $ \big\{ U^{A,B,C}_n \big\}_{n \in \mathbb N} \, $ is a basis of $C^{N \times N }[x]$ and $v$, $u^{A,B,C}$ have compact support, we get %\\
$v \equiv u^{A,B,C}$ but the inequality~\eqref{ine} is not possible. The uniform convergence follows from Stieltjes-Vitali theorem.
\end{proof}

\begin{cor}\label{coro3}
Under the hypotesis of Theorem~\ref{teorema 6} we have that,
%From~\eqref{limit} we get for all
%the sequences
$%\displaystyle
\big\{ Q_n \, Q_{n-1}^{-1} \, C_n^{-1} \big\}_{n \in \mathbb N} \, $ and
$%\displaystyle
\big\{ A_{n-1}^{-1}R_{n-1}^{-\top} \, R^{\top}_n \big\}_{n \in \mathbb N} \, $
uniformly converges to
$%\displaystyle
F_{A,B,C} \, $ on compact subsets \linebreak of
$ \mathbb C \setminus\bigcup_{k}\Gamma^{(k)}$.
%\begin{gather*}
%\lim_{n\to\infty} Q_n(x) \, Q_{n-1}^{-1}(x) \, C_n^{-1}=\lim_{n\to\infty} A_{n-1}^{-1}R_{n-1}^{-\top}(x) \, R^{\top}_n(x) = F_{A,B,C}(x) \, , \
% x\in\mathbb C \setminus\bigcup_{k}\Gamma^{(k)} \, .
%\end{gather*}
\end{cor}
Since the recurrence relations for
$\big\{ Q_n \big\}_{n\in\mathbb N} \, $
and
$\big\{ R_n \big\}_{n\in\mathbb N} \, $
can be~rewritten~as
\begin{align*}
x I_N =& A_n \, \big(
Q_{n+1}(x) \, Q^{-1}_n (x) \, C_{n+1}^{-1}
\big) \, C_{n+1} + B_n + (Q_n (x) \, Q^{-1}_{n-1}(x) \, C^{-1}_n)^{-1} \, , \\
x I_N =& A_n \, \big(
A_{n}^{-1} R^{-\top}_n (x) \, Q^{\top}_{n+1}(x)
\big) \, C_{n+1} + B_n + (A^{-1}_{n-1} R^{-\top}_{n-1}(x) \, Q^{\top}_n (x))^{-1} \, ,
\end{align*}
then the analytic function $F_{A,B,C} $ also satisfies a matrix equation
\begin{gather*}
A \, F_{A,B,C}(x) \, C \, F_{A,B,C}(x)+(B-xI_N ) \, F_{A,B,C}(x)+I_N =\pmb 0 \, .
\end{gather*}
\begin{cor}\label{coro4}
%Suppose that $A_n\to A \, $, $B_n\to B \, $, and $C_n\to C$ with $A, C$ nonsingular matrices, then
%\begin{gather*}
%\lim_{n\to\infty}R^{-\top}_{n}(x) \, V^{-1}_{n}(x)=F^{-1}_{C,B,A}(x)-A \, F_{A,B,C}(x) \, C \, , \ \ x\in\mathbb C \setminus\bigcup_{k}\Gamma^{(k)} \, ,
%% \\
%\end{gather*}
%\begin{gather*}
%\lim_{n\to\infty}G^{-\top}_{n}(x) \, Q^{-1}_{n}(x)=F^{-1}_{C,B,A}(x)-A \, F_{A,B,C}(x) \, C \, , \ \ x\in\mathbb C \setminus\bigcup_{k}\Gamma^{(k)} \, .
%\end{gather*}
Under the hypothesis of Theorem~\ref{teorema 6} we have that
the sequences
$%\displaystyle
\big\{ R^{-\top}_{n} \, V^{-1}_{n} \big\}_{n \in \mathbb N} \, $ and
$%\displaystyle
\big\{ G^{-\top}_{n} \, Q^{-1}_{n} \big\}_{n \in \mathbb N} \, $
uniformly converge on compact subsets of \linebreak $ \mathbb C \setminus \bigcup_{k}\Gamma^{(k)}$ to
$%\displaystyle
F^{-1}_{C,B,A}(x)-A \, F_{A,B,C}(x) \, C \, $.
\end{cor}

\begin{proof}
As a consequence of Christoffel-Darboux formulas when $x=y$ %we~get 	
\begin{gather*}
\big(
G_{n+1}(x) \, G_n^{-1}(x)
\big)^{\top} C_{n+1} -
A_{n} \, Q_{n+1}(x) \, Q_n^{-1}(x)=G_n(x)^{-\top}Q_n^{-1}(x) \\
%\end{gather*}
%\begin{gather*}
A_n \, V_{n+1}(x) \, V_n^{-1}(x)-
\big( R_{n+1}(x) \, R_n^{-1}(x)\big)^{\top} C_{n+1}=R_n(x)^{-\top}V_n^{-1}(x) \, ,
\end{gather*}
from here the limit follows.
\end{proof}

 When the sesquilinear form, $\langle { .} , { .} \rangle $, is associated with a positive definite symmetric matrix of measures, $\mu$, we have the representation
 \begin{gather*}
 \prodint{P(x),Q(x)}=\int P(x) \, d\mu \, Q^{\top}(x) \, .
 \end{gather*}
 Here we have orthonormality i.e. $V_n
 %(x)=
 \equiv G_n
 %(x)
 $ and they satisfy a recurrence~relation
 \begin{gather*}
 x \, V_n(x)=A_{n}V_{n+1}(x)+B_nV_n(x)+A^{\top}_{n-1}V_{n-1}(x) \, , \ \
 n\geq 0 \, ,
 \end{gather*}
with initial conditions
$ V_{-1}(x)={\bf 0} $
and
$ V_0(x)=I_N $,
$A_n$ nonsingular matrices and~$B_n$ Hermitian matrices. Thus, if
$ \big\{ V^{(k)}_{n} \big\}_{n\in\mathbb N} \, $ is the sequence of $k$-th associated
matrix polynomials which are orthonormal with respect to the matrix of measures~$d\mu^{(k)}$ and $A_n\to A$, $B_n\to B$ with $A$ a nonsingular matrix and $B$ a Hermitian matrix, then
\begin{gather*}
\lim_{k\to\infty}\int \frac{d\mu^{(k)}(x)}{z-x}=\int \frac{dW_{A^\top,B^{\top}}(x)}{z-x}
 % \, , \ \ z\in\mathbb C \setminus\bigcup_{k}\Gamma^{(k)}
\, ,
\end{gather*}
and the convergence holds uniformly on compact subsets of $\mathbb C \setminus\bigcup_{k}\Gamma^{(k)}$.

 Here, $dW_{A^\top,B^{\top}}$ denotes the matrix of measures for which the polynomials $U_n^{A^\top,B^{\top}}(x)$ defined by the recurrence formula
\begin{gather*}
 x \, U_n^{A^\top,B^{\top}}(x)=A \, U_{n+1}^{A^\top,B^{\top}}(x)+B \, U_n^{A^\top,B^{\top}}(x) + A^\top \, U_{n-1}^{A^\top,B^{\top}}(x), \ \ n\geq 0 \, ,
\end{gather*}
 are orthonormal.

 Moreover, if we assume that the matrix $A$ is positive definite and the matrix~$B$ is Hermitian, then in~\cite{Duran1} is showed that
 \begin{multline*}
\int \frac{dW_{A^\top,B^{\top}}(x)}{z-x}=\frac{1}{2} A^{-1} (zI_N -B) \, A^{-1}
 \\ -\frac{1}{2}A^{-1/2} \,
 \big(
 \sqrt{A^{-1/2} \, (B-zI_N ) \, A^{-1} ( B-zI_N ) \, A^{-1/2}}-4I_N
 \big) \, A^{-1/2} \, ,
 \end{multline*}
for $z\notin \operatorname{supp} \, (W_{A^\top,B^{\top}}) \, $.

\end{document}